\documentclass[perreview, 12pt]{elsarticle}

\usepackage{amsmath,amsthm}
\usepackage[latin1]{inputenc}
\usepackage{enumerate}

\usepackage{url}
\usepackage{color}
\usepackage{amsfonts}
\usepackage{subfigure}
\usepackage{algorithm}
\usepackage{algpseudocode}
\algnotext{EndFor}
\algnotext{EndIf}

\usepackage{graphicx}

\usepackage{epstopdf}

\usepackage[pdfpagelabels]{hyperref}
\hypersetup{colorlinks=true, linkcolor=blue, citecolor=blue, urlcolor=blue}

\newtheorem{remark}{Remark}

\newtheorem{definition}{Definition}

\newtheorem{lemma}[remark]{Lemma}
\newtheorem{theorem}[remark]{Theorem}
\newtheorem{proposition}[remark]{Proposition}
\newtheorem{corollary}[remark]{Corollary}

\newtheorem{observation}{Observation}

\newcommand\etal{\emph{et al.}}

\DeclareMathOperator{\adim}{adim}

\journal{Information Sciences}

\title{$k$-Metric Antidimension: a Privacy Measure for Social Graphs}

\author[a1]{Rolando Trujillo-Rasua \corref{cor1}}

\author[a2]{Ismael G. Yero}

\address[a1]{Interdisciplinary Centre for Security, Reliability and
Trust \\ University of Luxembourg \\
6, rue Richard Coudenhove-Kalergi
L-1359 Luxembourg}
\address[a2]{Departamento de Matem\'aticas, Escuela Polit\'ecnica Superior
de Algeciras \\ Universidad de C\'adiz, Spain \\
Av. Ram\'on Puyol s/n, 11202 Algeciras, Spain}

\cortext[cor1]{Corresponding author. Phone: +352 466 644 5458. Fax: 466 644 3
5458. Email:
rolando.trujillo@uni.lu}

\begin{document}

\begin{abstract}

The study and analysis of social
graphs impacts on a wide range of applications, such as community decision
making
support and recommender systems. With
the boom of \emph{online} social networks, such analyses are benefiting
from a massive collection and publication of social graphs at large scale.
Unfortunately, individuals' privacy right might be
inadvertently violated when publishing this type of data. In this
article, we introduce $(k, \ell)$-anonymity; a novel privacy
measure aimed at evaluating the resistance of
social graphs to active attacks. $(k, \ell)$-anonymity is based on a new
problem in Graph Theory, the \emph{$k$-metric antidimension} defined
as
follows.

Let $G = (V, E)$ be a simple connected graph and $S = \{w_1, \cdots,
w_t\}
\subseteq V$ an ordered subset of vertices. The metric representation of a
vertex $u\in V$ with respect to $S$ is the $t$-vector $r(u|S) = (d_G(u, w_1),
\cdots, d_G(u, w_t))$,
where $d_G(u, v)$ represents the length of a shortest $u-v$ path in $G$.
We call $S$ a $k$-antiresolving set if $k$ is the largest positive integer such
that
for every vertex $v \in V-S$ there exist other $k-1$ different vertices $v_1,
\cdots, v_{k-1} \in V-S$ with
$r(v|S) = r(v_1|S) = \cdots = r(v_{k-1}|S)$. The
$k$-metric antidimension of $G$ is the minimum cardinality among all the
$k$-antiresolving sets for $G$.

We address the $k$-metric antidimension problem by proposing a
true-biased algorithm with success rate above $80 \%$ when considering random
graphs of size at most $100$.
The proposed algorithm is used to determine the privacy guarantees offered
by two real-life social graphs with respect to $(k,
\ell)$-anonymity. We also
investigate theoretical properties
of the $k$-metric antidimension of graphs. In
particular, we focus on paths, cycles, complete bipartite graphs and trees.

\end{abstract}


%{\it AMS Subject Classification Numbers:}   05C12; 91D30; 05C82; 05C85.

\maketitle

\section{Introduction}

% % % % % % % % % % % % % % % % % % % % % % % % % % % % % % %
% Social networks; its popularity and consequences % % % % % % %% % % % % % % %
%%%%%%%%%%%%%%%%%%%%%%%%%%%%%%%%% % % % % % % % % % % % % % % % % % % % % % % %

Social networking services are widely used in modern society as
illustrated by the Alexa's Top 500 Global Sites
statistics\footnote{\url{http://www.alexa.com/topsites}} where
\emph{facebook} and \emph{linkedin} rank $2$nd and $11$th
respectively in 2014. Such
popularity has enabled governments and third-party enterprises to massively
collect social network data, which eventually can be
released\footnote{See for example \url{http://snap.stanford.edu/data/}} for
mining and
analysis purposes.

% % % % % % % % % % % % % % % % % % % % % % % % % % % % % % % % % % %
% Importance of Publishing social network data and performing analysis on it %
%%%%% % % % % % % % % % % % % % % % % % % % % % % % % % % % % %

The power of social network analysis is questionless. It might uncover
previously unknown knowledge such as community-based problem, media use,
individual engagement, amongst others. Sociology is a trivial example of a
field that certainly benefits from social
graphs publication. Many other fields (\emph{e.g.,} economics, geography, or
political science) and systems (\emph{e.g.,}
service-oriented systems, advertisers,
or recommended systems) improve
their
decisions, processes, and services, based on users interaction.

% % % % % % % % % % % % % % % % % % % % % % % % % % % % % % % % % % %
%Privacy problem and anonymization as a naive solution
% % % % % % % % % % % % % % % % % % % % % % % % % % % % % % % % % % %

However, all these benefits are not cost-free. An adversary can compromise
users privacy using the
published social network, which results in the disclosure of sensitive
data such as e-mails, instant messages, or relationships. A simple and
popular
approach to prevent this privacy problem is
\emph{anonymization} by means of removing potential
identifying
attributes. Doing so,
aggregate knowledge still can be inferred (\emph{e.g.,} connectivity,
distance, or node degrees) while the ``who'' information has been removed. In
practice, however, this naive approach is not enough for
protecting users' privacy.

% % % % % % % % % % % % % % % % % % % % % % % % % % % % % % % % %
% The challenge of the problem
% % % % % % % % % % % % % % % % % % % % % % % % % % % % % % % % % %

What makes social network anonymization a challenging problem is the combination
of the adversary's background knowledge with
the released structure of the network. Considering a social network as
a simple graph, in which individuals
are represented by vertices and their bidirectional relationships by edges,
the adversary's background knowledge about a victim may take many forms,
\emph{e.g.,} vertex degrees, connectivity, or local neighborhood. This
structural
knowledge, together with the released graph, is often enough to perform
\emph{passive attacks} where the users and their relationships are
re-identified~\cite{5207644}.

% % % % % % % % % % % % % % % % % % % % % % % % % % % % % % % % %
% Introducing active attacks
% % % % % % % % % % % % % % % % % % % % % % % % % % % % % % % % %

Other privacy attacks exist. In 2007, Backstrom
\etal~\cite{Backstrom:2007:WAT:1242572.1242598} introduced \emph{active attacks}
based on
the creation and insertion in the
network of \emph{attacker nodes} controlled by the adversary. The attacker
nodes could be either
new accounts
with pseudonymous or spoofed identities (Sybil nodes), or legitimate users in
the network who collude with the adversary. Attacker nodes establish links with
other nodes in
the network (also
between themselves) aiming at creating a sort of fingerprint in the network.
Once the social graph is released, the adversary just need to retrieve such a
fingerprint (the attacker nodes) and use it as a hub to re-identify other nodes
in the network. Backstrom et. al. proved that $O(\sqrt{\log n})$ attacker nodes
in the network
can compromise the privacy of arbitrary targeted nodes with high
probability, which makes active attack particularly dangerous.

\subsection{Contribution and plan of the article}

% % % % % % % % % % % % % % % % % % % % % % % % % % % % % % % % % % %
% Mentioning the limitations of current privacy analysis to active attacks. In
%particular, explaining that there does not exist any formal representation of
%the background knowledge and, consequently, no privacy metric exists neither.
% % % % % % % % % % % % % % % % % % % % % % % % % % % % % % % % % % % %

Several active attacks to social graphs have been proposed. They could even
target random nodes in the network as recently shown in \cite{6275441}.
However, to the best of our knowledge, no privacy measure aimed at evaluating
the resistance of a social graph to this kind of attack exists. The lack of
such a measure prevents the development of privacy-preserving methods with
theoretically proven privacy guarantees.

% % % % % % % % % % % % % % % % % % % % % % % % % % % % % % % % % % % % %
% First, we define the background knowledge as the "representative" in the
%resolving set problem.
% This allows us to define a privacy metric based on a degeneration of the
%resolving set problem, which is as follows: define the problem.
% Explains the consequences, which is, it is now possible to evaluate the
%resistance of social graphs to active attacks.
% % % % % % % % % % % % % % % % % % % % % % % % % % % % % % % % % % % %

In this article we define \emph{$(k, \ell)$-anonymity}; a privacy measure that
can be applied to
real-life social graphs in order to measure their resistance to active
attacks. The proposed privacy measure copes with adversaries whose background
knowledge
concerning a node $u$ and a subset $S$ of attacker nodes is the metric
representation of $u$ with respect to $S$. $(k, \ell)$-anonymity turns
out to be based on a new
problem in Graph Theory:
the \emph{$k$-metric antidimension}. We propose a
true-biased algorithm whose computational complexity and success rate can be
adjusted. Empirical results show that our algorithm finds
$k$-antiresolving basis in random graphs of order at most $100$ with a success
rate above $80 \%$. Our algorithm has been also used to
determine the privacy offered by two real-life social graphs against active
attacks. Finally, we
provide theoretical results on the
$k$-metric antidimension of graphs, such as paths, cycles, complete bipartite
graphs and trees.

% % % % % % % % % % % % % % % % % % % % % % % % % % % % % % % % % % % % %
% The plan of the article is as follows:
% 	- Related work
%	- Metric k dimension
%		- Background knowledge
%		- privacy metric
%	- Initial results
%	- Conclusions
% % % % % % % % % % % % % % % % % % % % % % % % % % % % % % % % % % % % %

The rest of this article is structured as follows.
Section~\ref{sec_related_work} briefly
reviews the literature on privacy-preserving publication of social network
data. Section~\ref{sec_antiresolving_sets} presents the metric representation as a
reasonable definition of the adversary's background knowledge. It also
introduces the $k$-metric antidimension as the basis for the privacy
measure $(k, \ell)$-anonymity. In Section \ref{sect-algo} we present a
true-biased algorithm for computing the $k$-metric antidimension of a
graph, and evaluate the proposed algorithm through experiments. Preliminary
results
(mathematical properties) on the new problem
(the $k$-metric antidimension) are provided in Sections
\ref{sec_metric_antidimension} and \ref{section-trees} (the later
specifically addresses the case of tree graphs).
Section~\ref{sec_conclusions} draws conclusions and future
work.

\section{Related work}\label{sec_related_work}

% % % % % % % % % % % % % % % % % % % % % % % % % % % % % % % %
% The defining the problem as a social graph publication problem and the
%adversary goals as either edge or vertex disclosure.
% % % % % % % % % % % % % % % % % % % % % % % % % % % % % % % % %

A social graph $G = (V, E)$ is a simple graph where $V$ represents the set of
social actors and $E \subseteq V \times V$ their relationships. Both vertices and
edges could be enriched with attribute values such as weights
representing trustworthiness or labels providing meaning. We consider, however,
social network data in its most ``simplest'' form, \emph{i.e.}, a simple graph without
further annotation.

% % % % % % % % % % % % % % % % % % % % % % % % % % % % % % % %
% Follow the structure of the paper "A survey of privacy preservation of
%graphs and social networks". In particular, explains active attacks and the
%different structural knowledge
% % % % % % % % % % % % % % % % % % % % % % % % % % % % % % % % %

Privacy breaches in social networks are mainly categorized in
identity disclosure or link disclosure~\cite{Wu2010}. To perform such attacks
adversaries rely on background knowledge, which is usually defined as
structural knowledge such as vertex degrees~\cite{Liu:2008:TIA:1376616.1376629} or neighborhoods~\cite{4497459}. The
assumptions on the adversary's background knowledge determine the
type of privacy attacks and the corresponding countermeasures.

Privacy-preserving methods for the publication of social graphs are normally
based on the well-known concept
$k$-anonymity~\cite{Samarati98protectingprivacy} adapted to graphs.
$k$-anonymity, initially proposed for microdata, aims at ensuring that no
record in a
database can be re-identified with probability higher than $1/k$. To do so,
identifying attributes should be obviously removed, and any combination of
non-identifying attribute values should not be unique in the database. In
practice, not all the attributes need to be combined, because they do not
belong to the adversary's knowledge. This leads to the concept of
\emph{quasi-identifier}, that is, an attribute that can be found in external
source of information and, combined with other quasi-identifiers, can uniquely
identify a record in the database.

Even though graphs can be represented in tabular form and, thus,
graph $k$-anonymity can be defined in terms of quasi-identifying
attributes~\cite{DBLP:journals/soco/StokesT12}, graph $k$-anonymity is
typically defined in terms of structural properties of the graph rather than on
attributes. For instance in~\cite{Hay:2008:RSR:1453856.1453873}, the
adversary's background knowledge
is defined as a \emph{knowledge query} $Q(x)$ evaluated for a
given target node of the original graph $G$. The knowledge query $Q(x)$ allows
the
creation of a candidate set consisting of $\{y \in V | Q(x) = Q(y)\}$. In other
words, all the nodes in the network matching the query $Q(.)$ are equally
likely to be the target node $x$. This simple concept is the basis of several
passive attacks and privacy-preserving methods in the publication of social
graphs~\cite{Liu:2008:TIA:1376616.1376629,4497459,Zou:2009:KGF:1687627.1687734}.

Other privacy notions based on entropy rather than on $k$-anonymity have been
proposed~\cite{entropy}. This type of privacy measure is better suited for
methods based on random addition, deletion, or switching of edges. The
perturbation could be made in such a way that the number of edges or the
degree of the vertices are preserved~\cite{DBLP:conf/sdm/YangW08,randomizing}.
However, empirical results obtained
in~\cite{Ying:2009:CRK:1731011.1731021,randomizing}
suggest that random obfuscation
poorly preserves the topological features of the network.

% % % % % % % % % % % % % % % % % % % % % % % % % % % % % % % % % %5
% Explain Active attacks:
% 	- They target neighbors only and this is problematic since edges are
% difficult to create. See "De-anonymizing social networks".
% 	- There exists detection mechanism for active attacks ("Sybil-limit").
% However, they do not properly work in real world networks and are not false
% positive free.
% % % % % % % % % % % % % % % % % % % % % % % % % % % % % % % % % % % %

Passive attacks to social networks can be combined with active attacks. In
addition to structural
knowledge, in an active attack the
adversary manages to control a subset of nodes (attacker nodes) of the
original graph $G$~\cite{Backstrom:2007:WAT:1242572.1242598}.  The attacker
nodes aim at
creating links with their
victims by either identity theft or cloning of existing users
profiles~\cite{Bilge:2009:YCB:1526709.1526784}. They also establish
links between themselves so as to build a subgraph $H$ of attacker nodes with
the following properties: i) $H$ can be efficiently identified in $G$ and ii) $H$ does not have a non-trivial automorphisms. Once $H$ has been
identified, the adversary is able to re-identify neighbor nodes of $H$~\cite{Backstrom:2007:WAT:1242572.1242598} or even
arbitrary nodes in the network~\cite{5207644,6275441}.

Performing active attacks is not easy, given that there exist
several detection
mechanisms
of attacker or Sybil nodes in a
network~\cite{Yu:2010:SNS:1959337.1959353}. However, such defenses
strongly depend on assumptions on the topological
structure of the social network, which does not hold in many
real-world scenarios~\cite{Mohaisen:2011:MFF:1968613.1968648}. Actually,
recent works aim at mitigating, instead of preventing, the impact of
Sybil attacks~\cite{Viswanath:2012:CSS:2168836.2168867}. Furthermore, a group
of users who collude in
order to breach the privacy of other users in the network can be also regarded
as attacker nodes.

Other types of active attacks exist. For
instance, the maximal vertex coverage (MVC) attack consists in attacking a few
nodes so as to delete as many edges of the network as possible. In this attack,
the attacker tries to convince some users to
leave the social network in order to reduce the number of residual
social ties. Metrics to quantify the impact of MVC attacks have been studied
in~\cite{mvc}. MVC is not a privacy attack, though.

% % % % % % % % % % % % % % % % % % % % % % % % % % % % % % % % %
% Ending again with the contributions of our work
% % % % % % % % % % % % % % % % % % % % % % % % % % % % % % % % %5

While there exist several published active attacks to social graphs, there
does not exist yet a rational privacy metric for evaluating the resistance
of social graphs to this type of privacy attack. To overcome this problem, in
this
article we
introduce $(k, \ell)$-anonymity; a privacy notion based on $k$-anonymity and
the metric
representation of nodes in a graph. Note that, privacy notions with the
same name has been already proposed. For instance, Feder and Nabar proposed
$(k, \ell)$-anonymity where $\ell$ represents the number of common neighbors of
two nodes~\cite{DBLP:journals/corr/abs-0810-5578}. This notion was later
generalized by Stokes and Torra
in~\cite{DBLP:journals/soco/StokesT12}. In our privacy notion, however, $\ell$
represents an upper bound on the $k$-metric antidimension of the graph.

Interested readers could
refer to~\cite{6113304,Wu2010,Zhou:2008:BSA:1540276.1540279}
for further reading on privacy-preserving publication of social graphs.

\section{Privacy against active attacks}\label{sec_antiresolving_sets}

% % % % % % % % % % % % % % % % % % % % % % % % % % % % % % % % %
% Briefly explain the structure and aim of this section and its
% subsections.
% % % % % % % % % % % % % % % % % % % % % % % % % % % % % % % % % %

In this section we define the metric representation of nodes with
respect to a set of attacker nodes $S$ as the adversary's background knowledge.
We also introduce the resulting privacy measure, named $(k, \ell)$-anonymity,
and
its related mathematical problem: the $k$-metric
antidimension.

\subsection{Adversary's background knowledge}

% % % % % % % % % % % % % % % % % % % % % % % % % % % % % % % % % % %
% Defining background knowledge and how does this relate to current approaches.
% % % % % % % % % % % % % % % % % % % % % % % % % % % % % % % %

Vulnerabilities in an anonymized social graphs are better understood once
the adversary's knowledge has been properly modeled. This knowledge can be
acquired from public information sources and through malicious actions. In
practice, the adversary could even be a close friend, which makes
the publication of social network where users cannot re-identify themselves a
reasonable privacy goal.

Adversary's background information in passive attacks is typically modeled as
structural knowledge on the network. This is a sort of \emph{global} view that
provides
adversaries with the ability to partition the set of nodes into equivalence
classes of structurally equivalent nodes. The strongest of those structural
relations is automorphism~\cite{Zou:2009:KGF:1687627.1687734}. Two
vertices $u$ and
$v$ are automorphically
equivalent if
there exists an isomorphism from the graph
to itself such that $u$ maps to $v$. Other types of structural relations are
based on vertex degrees, connectivity, or local
neighborhood. Intuitively, structurally equivalent vertices are
indistinguishable with respect to the considered structural property.

However, adversaries controlling attacker nodes in a network
are undoubtedly more powerful. In addition to the global view, they
have a \emph{local view} determined by the relationship of the attacker
nodes with the network. To illustrate this let us consider the graph shown
in Figure~\ref{fig:graph:path}. With respect to the vertex degree property,
$v_2$ and $v_3$ are indistinguishable. They are easily re-identifiable by
either an adversary or a legitimate user owning the vertex $v_4$ and knowing
its distance to $v_2$ and $v_3$, though.

\begin{figure}
\centering
\includegraphics[width=0.45\textwidth]{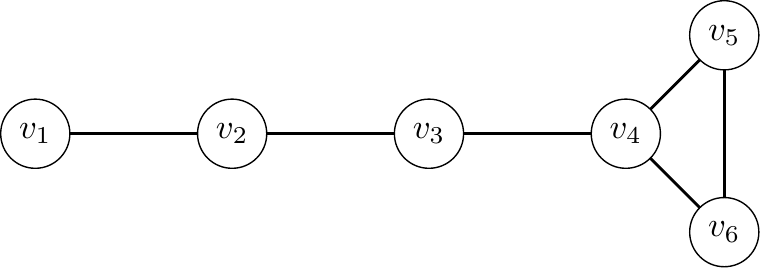}
\caption{An example.}
\label{fig:graph:path}
\end{figure}

A first step towards modeling such local view
was given by Hay et al.~\cite{Hay:2008:RSR:1453856.1453873}, who defined
the concept of \emph{hub fingerprint queries}. A hub is
a relevant node in the network with high degree and high centrality, and a hub
fingerprint for a target node $x$ is a vector of distances from $x$ to hub
vertices. Although not explicitly mentioned
in~\cite{Hay:2008:RSR:1453856.1453873}, the largest hub fingerprint for a
target node $x$
is indeed the \emph{metric representation} of $x$ with respect to the hub
vertices. We
formally define this concept as follows.

\begin{definition}[Metric representation]\label{def_vertex_representation}
Let $G = (V, E)$ be a simple connected graph and $d_G(u, v)$ be the length of the shortest
path between the vertices $u$ and $v$ in $G$. For an ordered set $S =\{u_1,\cdots, u_{t}\}$ of vertices in $V$ and a vertex $v$, we call $r(v|S) = (d_G(v, u_1), \cdots, d_G(v, u_{t}))$ the \emph{metric representation} of $v$ with respect to $S$.
\end{definition}

Similarly to Hay et al. work~\cite{Hay:2008:RSR:1453856.1453873},
we define the
adversary's background knowledge about a target
node $u$ as the metric representation of $u$ with respect to $S$. In this
article, however, we assume $S$ to be any subset of attacker nodes rather than
hub vertices only.
% % % % % % % % % % % % % % % % % % % % % % % % % % % % % % %
%Starting paragraph for first revision
% % % % % % % % % % % % % % % % % % % % % % % % % % % % % % %
%We say that two vertices $u, v \in V$ are indistinguishable with respect to a
%subset of attacker nodes $S \subseteq V$ if and only if $r(u | S) = r(v|S)$.
%This indistinguishability notion is actually less restrictive than
%\emph{automorphic equivalence}~\cite{Hay:2008:RSR:1453856.1453873} as it is
%shown in Theorem~\ref{theo:automorphism} below.
%
%\begin{theorem}\label{theo:automorphism}
%Two vertices $u$ and $v$ are
%automorphically equivalent only if for every $S \subseteq V-\{u,v\}$ it holds
%that $r(u |
%S) = r(v| S)$.
%\end{theorem}
%
%\begin{proof}
%Let us consider
%two vertices $u, v \in V$ automorphically
%equivalent. Recall that $u$ and $v$ are said to be \emph{automorphically
%equivalent} if
%there exists an isomorphism $f: V \rightarrow
%V$ from the graph
%to itself such that $f(u) = v$ and $f(v) =
%u$~\cite{Hay:2008:RSR:1453856.1453873}. Let us assume the existence of a
%subset
%of
%vertices $S \in V-\{u, v\}$ such that $r(u | S) \neq r(v | S)$. This means
%that there exists $x \in S$ such that $d_G(u, x) \neq d_G(v, x)$.
%\end{proof}
% % % % % % % % % % % % % % % % % % % % % % % % % % % % % % %
%Finishing paragraph for first revision
% % % % % % % % % % % % % % % % % % % % % % % % % % % % % % %

It is worth mentioning that the concept of metric representation is also the
basis of two weel-known concepts: resolving sets and metric dimension (cf.
Definition~\ref{def:resolving}). Both
have been already motivated
by  problems related to unique recognition of an intruder position in a
network~\cite{Slater1975}, where resolving sets were called \emph{locating
sets}. The name ``resolving set'' is
due to Harary and Melter \cite{Harary1976}, who introduced the concept
in 1976.

%The concept of resolving sets has been recently
%extended to $k$-resolving sets~\cite{Estrada-Moreno2013,Yero2013c}. Those
%problems, however, are different in nature to the
%problem introduced in this article.

\begin{definition}[Resolving set and metric dimension]\label{def:resolving}
Let $G = (V, E)$ be a simple connected graph. A set $S\subset V(G)$ is said to be a
\emph{resolving set} for $G$ if any pair of vertices of $G$ have different
metric representations with respect to $S$. A resolving set of the smallest
possible cardinality is called a \emph{metric basis}, and its cardinality the
\emph{metric dimension} of $G$.
\end{definition}

\subsection{$(k, \ell)$-anonymity}

% % % % % % % % % % % % % % % % % % % % % % % % % % % % % % % % %5
% Defining the metric and how does this relate to previous metrics
% % % % % % % % % % % % % % % % % % % % % % % % % % % % % % % % % % %5

$(k, \ell)$-anonymity is a privacy measure that evolves from the
adversary's background knowledge defined previously. It is based on the concept of
\emph{$k$-antiresolving set} defined as follows.

\begin{definition}[$k$-antiresolving set]\label{def_antiresolving_set}
Let $G = (V, E)$ be a simple connected graph and let $S = \{u_1,\cdots, u_{t}\}$ be a subset of
vertices of $G$. The set $S$ is called a $k$-\emph{antiresolving set} if $k$ is the greatest positive integer such that
for every vertex $v \in V-S$ there exist at least $k-1$ different vertices $v_1, \cdots, v_{k-1} \in V-S$ with
$r(v|S) = r(v_1|S) = \cdots = r(v_{k-1}|S)$, \emph{i.e.}, $v$ and $v_1, \cdots, v_{k-1}$ have the same metric representation with respect to $S$.
\end{definition}

The following concepts derive from Definition~\ref{def_antiresolving_set},
whose study is one of the goals of this article.

\begin{definition}[$k$-metric antidimension and $k$-antiresolving basis]\label{def_metric_antidimension}
The $k$-\emph{metric antidimension} of a simple connected graph $G = (V, E)$ is the
minimum cardinality amongst the $k$-antiresolving sets in $G$ and is denoted by $\adim_k(G)$. A $k$-antiresolving set of cardinality $\adim_k(G)$ is called a $k$-\emph{antiresolving basis} for $G$.
\end{definition}

It is easy to prove that if the set of attacker nodes $S$ is a
$k$-antiresolving set, the adversary cannot uniquely re-identify
other nodes in the network with probability higher than $1/k$. However, given
that $S$ is unknown, the privacy measure should quantify over all possible
subsets $S$ as follows.

\begin{definition}[$(k, \ell)$-anonymity]\label{def_graph_anonymity}
A graph $G$ meets $(k, \ell)$-\emph{anonymity} with respect
to active attacks if $k$ is the smallest positive integer such that the
$k$-metric antidimension of $G$ is lower than or equal to $\ell$.
\end{definition}

% % % % % % % % % % % % % % % % % % % % % % % % % % % % % % % % %5
% Provide rationality of the thresholds
% % % % % % % % % % % % % % % % % % % % % % % % % % % % % % % %

In Definition~\ref{def_graph_anonymity} the parameter $k$ is used as a privacy
threshold, whilst $\ell$ is an upper bound on the expected number of attacker
nodes in the network. Because attacker nodes are
difficult to
enrol in a network without been detected~\cite{Yu:2010:SNS:1959337.1959353},
$\ell$ can be estimated through
statistical analysis. A fair assumption, for example, is that the number of
attacker nodes is significantly lower than
the total number of nodes in the network. To further explain the role of $k$
and $\ell$ in
Definition~\ref{def_graph_anonymity}
we provide the following example result.

\begin{theorem}\label{def_simple_theorem}
For every $n > 0$ and $0 < \ell < n$, the graph $K_n$ meets $(n -
\ell,\ell)-anonymity$.
\end{theorem}

\begin{proof}
Since all the vertices in a complete graph $K_n$ are connected, every subset
$S$ of vertices of
$K_n$ is an $(n-|S|)$-antiresolving set. Therefore, the $k$-metric antidimension
of $K_n$ is $n-k$.

According to Definition~\ref{def_graph_anonymity}, the $k$-metric antidimension
should be lower than or equal to $\ell$, which implies that $k \geq n - \ell$.
Moreover, $k$
should be the smallest positive integer satisfying the previous condition.
Therefore, $K_n$ holds $(n - \ell, \ell)$-\emph{anonymity}.
\end{proof}

\begin{corollary}\label{def_simple_corollary}
A social graph $K_n$ guarantees that a user cannot be re-identified
with probability higher than $\frac{1}{n - \ell}$ by an adversary controlling
$\ell$ attacker nodes.
\end{corollary}

These simple and intuitive result obtained in Theorem~\ref{def_simple_theorem}
and Corollary~\ref{def_simple_corollary} shows the role of the privacy measure
$(k,
\ell)$-anonymity in privacy-preserving publication of social graphs. Before
releasing a social graph $G$, the goal is to find $k$ such that $G$
satisfies $(k,\ell)$-anonymity. To do so, theoretical results and efficient
algorithms on the $k$-metric antidimension of a graph need to be investigated.

%It is worth remarking that $(k,\ell)$-anonymity might not be enough to protect
%users' privacy. This privacy measure is intended to deal with active
%attacks only. An adversary with structural knowledge, such as the degree of
%the targeted nodes, can also perform passive attacks. In this case,

\section{Computing the $k$-metric antidimension}\label{sect-algo}

Computing the $k$-metric antidimension of a graph seems to be a challenging
problem, whose hardness ought to be investigated. It should be
remarked that its counterpart: the $k$-metric dimension is an NP-complete problem \cite{Yero2013c} (and \cite{Khuller:1996:LG:244612.244621} for $k=1$).
Particularly, we observe that any resolving set is also a $1$-antiresolving set, which gives some intuitive idea on the hardness of computing the $k$-metric antidimension of graphs. Thus, we address the $k$-metric
antidimension problem by proposing a true-biased algorithm whose
success rate and computational cost can be balanced.

\subsection{A true-biased algorithm}

A true-biased algorithm is always correct when it returns \texttt{true},
it might fail with some small probability when its output is \texttt{false}.
True-biased algorithms normally are Monte Carlo algorithms with deterministic
running time and randomized behavior. The algorithm we introduce in this
section resembles to a
Monte Carlo algorithm in the sense that it is deterministic and has the
true-biased property. The proposed algorithm is not randomized, though.

The mathematical foundation of our algorithm requires the introduction of
notation as follows. For a given subset of vertices $X \subseteq V(G)$, we
denote $\sim_X: V(G) \times V(G)$ to the symmetric, reflexive and
transitive relation satisfying that $u \sim_X v \implies r(u | X) = r(v | X)$.
The set of equivalence classes created by $\sim_X$ over the subset of vertices
$V(G) - X$ is denoted as $C_{X}$. We deliberately abuse notation and
use $\sim_v$ and $C_v$ instead of $\sim_{\{v\}}$ and $C_{\{v\}}$ for every
vertex $v \in V(G)$.

%\begin{proposition}
%A set $S' \subseteq V(G)$ is a potential $k$-antiresolving basis only if $S'$
%is a $k'$-antiresolving set and $k' > k$.
%\end{proposition}

\begin{proposition}\label{prop_1}
Let $S \subseteq V(G)$ and $S' \subseteq S$:
\begin{itemize}
	\item $u \sim_S v \implies u \sim_{S'} v$
	\item $\forall 	X \in C_{S}$  there exists $X' \in C_{S'}$ such that $X
	\subseteq X'$
	\item $\forall 	X \in C_{S}$ and  $\forall X' \in C_{S'}$, $X \cap
	X' \neq \emptyset \implies X \subseteq X'$
\end{itemize}
\end{proposition}

\begin{lemma}\label{theo_1}
Let $S$ be a $k$-antiresolving  set and let $S' \subseteq S$. Let $Y = \{X \in
C_{S'}\,:\, | X| < k\}$, then $S'\cup (\bigcup_{y \in Y} y) \subseteq S$.
\end{lemma}

\begin{proof}
By Proposition~\ref{prop_1}, for every $X \in
C_{S}$ there exists $X' \in
C_{S'}$ such that $X \subseteq X'$, which implies that $|X'| \geq |X|
\geq k$ due to the definition of $k$-antiresolving set. Consequently, $|X'| <
k$ implies that there does not exist $X \in
C_S$ such that $X \subseteq X'$, meaning that there does not exist $X \in
C_S$ such that $X \cap X' \neq \emptyset$ according to
Proposition~\ref{prop_1}. Therefore, $X' \cap (V(G)-S) = \emptyset$ and thus
$X' \subseteq S$.
\end{proof}

In the spirit of Lemma~\ref{theo_1}, let $f: V(G) \rightarrow V(G)$ be the
function defined recursively as follows:

\begin{equation}\label{eq_1}
f(S) = \begin{cases} f(S\cup (\bigcup_{y \in Y} y)), & \mbox{if } Y = \{X \in
C_{S}\,:\, |X| < k\} \mbox{ is not empty}, \\
S, & \mbox{otherwise}.
\end{cases}
\end{equation}

According to Lemma~\ref{theo_1}, if $S$ is a subset of a
$k$-antiresolving set, so is $f(S)$. We therefore give some useful properties
of the function $f$ in Theorem~\ref{theo_3} below.

\begin{theorem}\label{theo_3}
The function defined in Equation~\ref{eq_1} satisfies the following
properties.
\begin{enumerate}
	\item $f(f(S)) = f(S)$
	\item $S' \subseteq S \implies f(S') \subseteq f(S)$
	\item $\forall S' \subset S, f(S) = f(f(S-S')\bigcup f(S'))$
	\item $S' \subseteq f(S) \implies f(S') \subseteq f(S)$
\end{enumerate}
\end{theorem}

\begin{proof}
The first property comes straightforwardly from Equation~\ref{eq_1}. In order
to prove the second property, let $S' \subseteq S$ and $u \in f(S')$.
If $u \in S$, then $u \in f(S)$ by
definition. Let us thus assume that $u \notin
S$. Given that $u \in f(S')$, there exist $X' \in C_{S'}$ such that $|X'| < k$
and $u \in X'$. Let $X \in C_{S}$ such that $u \in X$. Note that, such an $X$
exists because $u \notin S$. According to
Proposition~\ref{prop_1}, since $X \cap X' \neq \emptyset$ and $S' \subseteq
S$, then $X \subseteq X'$, which means that $|X|
< k$ and that $X \subseteq f(S)$, which proves the second property.

The third property can be proven by using the first property. Given that $S -
S' \subset S$ and $S' \subset S$, then $f(S - S') \subseteq f(S)$ and $f(S')
\subseteq f(S)$, hence, $f(f(S - S') \bigcup f(S')) \subseteq f(S)$. Similarly,
$S-S' \subseteq f(S-S')$ and $S' \subseteq f(S')$ by definition, which implies
that $S \subseteq f(S-S') \bigcup f(S')$. Again, applying the first property we
obtain that $f(S) \subseteq f(f(S-S') \bigcup f(S'))$. The two results lead to
$f(S) = f(f(S-S') \bigcup f(S'))$.

Finally, the last property is proven as follows. If $S' \subseteq
f(S)$, then $f(S') \subseteq f(f(S))$ by applying the second property. The
proof is concluded by simply considering the first property.
\end{proof}

The function $f(.)$ is the basis of Algorithm~\ref{alg_antiresolving}, which
aims to find a $k$-antiresolving set in a graph.
Algorithm~\ref{alg_antiresolving} is an optimized version supported by
Theorem~\ref{theo_3} of the following algorithm. Let us consider all subsets
$S$ of $V(G)$ with cardinality lower than or equal to $m$. If $f(S)$ is a
$k$-antiresolving set, then
a positive output is provided. If not,
a proof that a $k$-antiresolving set does not exist is found when
$f(S) = V(G)$ for every $S \subseteq V(G)$ such that $|S| = m$. Note that, this
impossibility result comes from the monotonicity of the function $f$,
\emph{i.e.,} $S' \subseteq S \implies f(S') \subseteq f(S)$. Any other case
leads to the $\emph{unknown}$ state where neither a proof nor a disproof of
the existence of a $k$-antiresolving set can be found.

\begin{algorithm}
\caption{Given a positive integer $k$, this
algorithms outputs: i) $\texttt{true}$ if it finds a
$k$-antiresolving set, ii) $\texttt{false}$ if such a set does not
exist, iii) $\texttt{unknown}$ when neither a $k$-antiresolving set nor a
proof that such a set does no exist was found.
\label{alg_antiresolving}}
\begin{algorithmic}[1]
\Require A graph $G$, an integer value $m$ to control the exponential
explosion, and the integer value $k$.
\State Let $V(G) = \{v_1, \cdots, v_N\}$
\State Let $C_1 = \{f(\{v_1\}), \cdots, f(\{v_N\})\}$ \label{step_1}
\If {$\exists S \in C_1$ that is a $k$-antiresolving set}
	\Return true \label{alg:line:return:true:1}
\EndIf
\For {$h = 2$ to $m$}\label{step_2}
	\State Let $C_h$ be an empty set
	\For {$i = 1$ to $|C_{h-1}|$}
		\State Let $S_i$ be the $i$th element of $C_{h-1}$
		\For {$j = i+1$ to $|C_{h-1}|$}
			\State Let $S_j$ be the $j$th element of $C_{h-1}$
			\If {$S_i \not\subseteq S_j$ and  $S_j \not\subseteq
			S_i$}\label{step_3}
				\State $S = f(S_i \bigcup S_j)$
				\If {$S$ is a $k$-antiresolving set}
					\Return \texttt{true}\label{alg:line:return:true:2}
				\EndIf
				\State Add $S$ to $C_h$
			\EndIf
		\EndFor
	\EndFor
\EndFor
\If {$\forall S \in C_m, S = V(G)$}
	\Return \texttt{false}
\Else {}
	\Return \texttt{unknown}
\EndIf
\end{algorithmic}
\end{algorithm}

Algorithm~\ref{alg_antiresolving} can be considered a true-biased algorithm if
the unknown state is regarded as a negative result. Its computational
complexity is clearly exponential
in terms of $m$. More precisely, for every $i \in \{2, \cdots, m\}$ we obtain
that $|C_{i-1}| \leq |C_i| \leq |C_{i-1}|(|C_{i-1}|-1)/2$, because $C_i$ is
formed by joining
every pair of elements of $C_{i-1}$. This means that the computational
complexity of Algorithm~\ref{alg_antiresolving} is determined by the size of
$C_m$. Given that, in the worst case, the cardinality of $C_m$ quadratically
increases with respect to $C_{m-1}$, we obtain that the worst-case
computational
complexity of this algorithm is $\mathcal{O}(N^{2^{m-1}})$.

Although $\mathcal{O}(N^{2^{m-1}})$ is double exponential in terms of $m$, when
$m << N$ it becomes
significantly lower than the computational complexity of a brute force
algorithm that considers the $2^{N}$ subsets of $V(G)$. For example, for $m =
1$, $m = 2$, and $m = 3$, the computational complexity becomes
$\mathcal{O}(N)$,
$\mathcal{O}(N^2)$, and $\mathcal{O}(N^4)$, respectively. Moreover, given
that the search space monotonically increases with $m$,
the accuracy of the algorithm also increases with $m$. In this sense, $m$
provides a trade-off between false negatives and computational cost.

It is worth remarking
that a theoretical lower bound, although not considered in the analysis, of the
computational complexity of Algorithm~\ref{alg_antiresolving} is
$\mathcal{O}(N^3)$, which is the computational complexity of the \emph{classic}
Floyd-Warshall
algorithm required to compute the metric representation of all
vertices. This prevents our method to be used on large graphs even when $m =
1$. In this case, more efficient implementations of both
Algorithm~\ref{alg_antiresolving} and the Floyd-Warshall algorithm ought to be
considered, \emph{e.g.,}~\cite{GBSW2010}.

Algorithm~\ref{alg_antiresolving} can be adapted to find a
$k$-antiresolving basis rather than a $k$-antiresolving set. To that aim, we
rely
on Proposition~\ref{prop_2} below. Proposition~\ref{prop_2} gives a
sufficient condition for the presence of a $k$-antiresolving basis. This
implies just a small modification to Algorithm~\ref{alg_antiresolving}. In
particular, the conditional statements in lines~\ref{alg:line:return:true:1}
and~\ref{alg:line:return:true:2} should take into account that such sufficient
condition is satisfied. The full pseudo-code considering this modification is
presented in
Algorithm~\ref{alg_antiresolving_basis}.

\begin{proposition}\label{prop_2}
Let $S$ be the subset of smaller cardinality in $V(G)$
such that $f(S)$ is a $k$-antiresolving set. Then, $f(S)$ is a
$k$-antiresolving basis if $\forall S' \subseteq V(G)$ such that $|S'| = |S|$
it follows that $|f(S)| \leq |f(S')|$.
\end{proposition}

\begin{algorithm}
\caption{Given a positive integer $k$, this
algorithms outputs: i) \texttt{true} if a
$k$-antiresolving basis is found, ii) $\texttt{false}$ if a $k$-antiresolving
basis does not
exist, iii) $\texttt{unknown}$ when neither a $k$-antiresolving basis nor a
proof that it does not exist was found.
\label{alg_antiresolving_basis}}
\begin{algorithmic}[1]
\Require A graph $G$, an integer value $m$ to control the exponential
explosion, and the integer value $k$.
\State Let $V(G) = \{v_1, \cdots, v_N\}$
\State Let $C_1 = \{f(\{v_1\}), \cdots, f(\{v_N\})\}$
\State Let $minSet = \min(|f(\{v_1\})|, \cdots, |f(\{v_N\})|)$
\If {$\exists S \in C_1$ such that $S$ is a $k$-antiresolving set and $|S|
\leq minSet$}
	\Return \texttt{true}
\EndIf
\For {$h = 2$ to $m$}
	\State Let $C_h$ be an empty set
	\For {$i = 1$ to $|C_{h-1}|$}
		\State Let $S_i$ be the $i$th element of $C_{h-1}$
		\For {$j = i+1$ to $|C_{h-1}|$}
			\State Let $S_j$ be the $j$th element of $C_{h-1}$
			\If {$S_i \not\subseteq S_j$ and  $S_j \not\subseteq S_i$}
				\State Add $f(S_i \bigcup S_j)$ to $C_h$
			\EndIf
		\EndFor
	\EndFor
	\State Let $minSet$ be the minimum cardinality of a set in $C_h$
	\If {$\exists S \in C_1 \cup \cdots \cup C_h$ such that $S$ is a
	$k$-antiresolving set and
	$|S| \leq minSet$}
		\Return \texttt{true}
	\EndIf
	\EndFor
\If {$\forall S \in C_m, S = V(G)$}
	\Return \texttt{false}
\Else {}
	\Return \texttt{unknown}
\EndIf
\end{algorithmic}
\end{algorithm}

\subsection{Empirical evaluation on synthetic graphs}

In order to show the feasibility of both Algorithm~\ref{alg_antiresolving} and
Algorithm~\ref{alg_antiresolving_basis}, we
ran
experiments
considering $m \in \{1, 2, 3\}$ and random
graphs as input. The aim of the experiments is to provide statistically sound
data on
the ratio between positive, negative, and unknown results of the
proposed algorithms. Further below in this section we also show results on
real-life
social graphs.

A random graph is created by choosing
integer values uniformly distributed in the interval $[k+2,100]$ as the
number of
vertices $N$; where $k$ is a privacy threshold. The number of edges also
distributes uniformly in the
interval
$[0,N\times(N-1)/2]$, and the edges are added randomly to the graph. For
each pair $(m,k) \in \{1, 2, 3\} \times \{1, 2, 3, 4, 5, 6, 7, 8\}$, we
created $10\ 000$ random graphs and executed Algorithm~\ref{alg_antiresolving}
and Algorithm~\ref{alg_antiresolving_basis} in order to look for a
$k$-antiresolving set and a $k$-antiresolving basis, respectively.

The success rate of both algorithms considering $m \in \{1,2,3\}$ is shown in
Figure~\ref{fig:rates}. We define a success as either a positive or a negative
result, \emph{i.e.,} whenever a $k$-antiresolving set (basis) or a proof that
it does
not exist is found. According to
Figure~\ref{fig:rates}, both algorithms perform poorly for $m = 1$. However,
when $m = 2$ they already achieve a success rate above $80 \%$, which
is further improved by the more
computationally demanding versions of Algorithm~\ref{alg_antiresolving} and
Algorithm~\ref{alg_antiresolving_basis} that consider $m = 3$.

The difference between Figure~\ref{figure:set} and
Figure~\ref{figure:basis} suggests, as expected, that finding a
$k$-antiresolving basis is harder than finding
a $k$-antiresolving set. Notwithstanding,
Algorithm~\ref{alg_antiresolving_basis} performs above $80 \%$ when $m = 2$ and
$m = 3$. It is also worth remarking that, even though Figure~\ref{fig:rates}
hints
that the success rate of both algorithms monotonically decreases with $k$, our
algorithms
have $100 \%$ of success rate if $k$ is equal to the order of the graph. This
is because $\forall u \in V(G)\,  \forall  X \in
C_u (|X|<|V(G)|)$ and, thus, all the nodes in the graph should be contained in a
$k$-antiresolving set according to Lemma~\ref{theo_1}.

\begin{figure*}
\hspace{-15pt}
\centering
  \subfigure[$k$-antiresolving set]{
       \includegraphics[width=0.35\textwidth,
       angle=270]{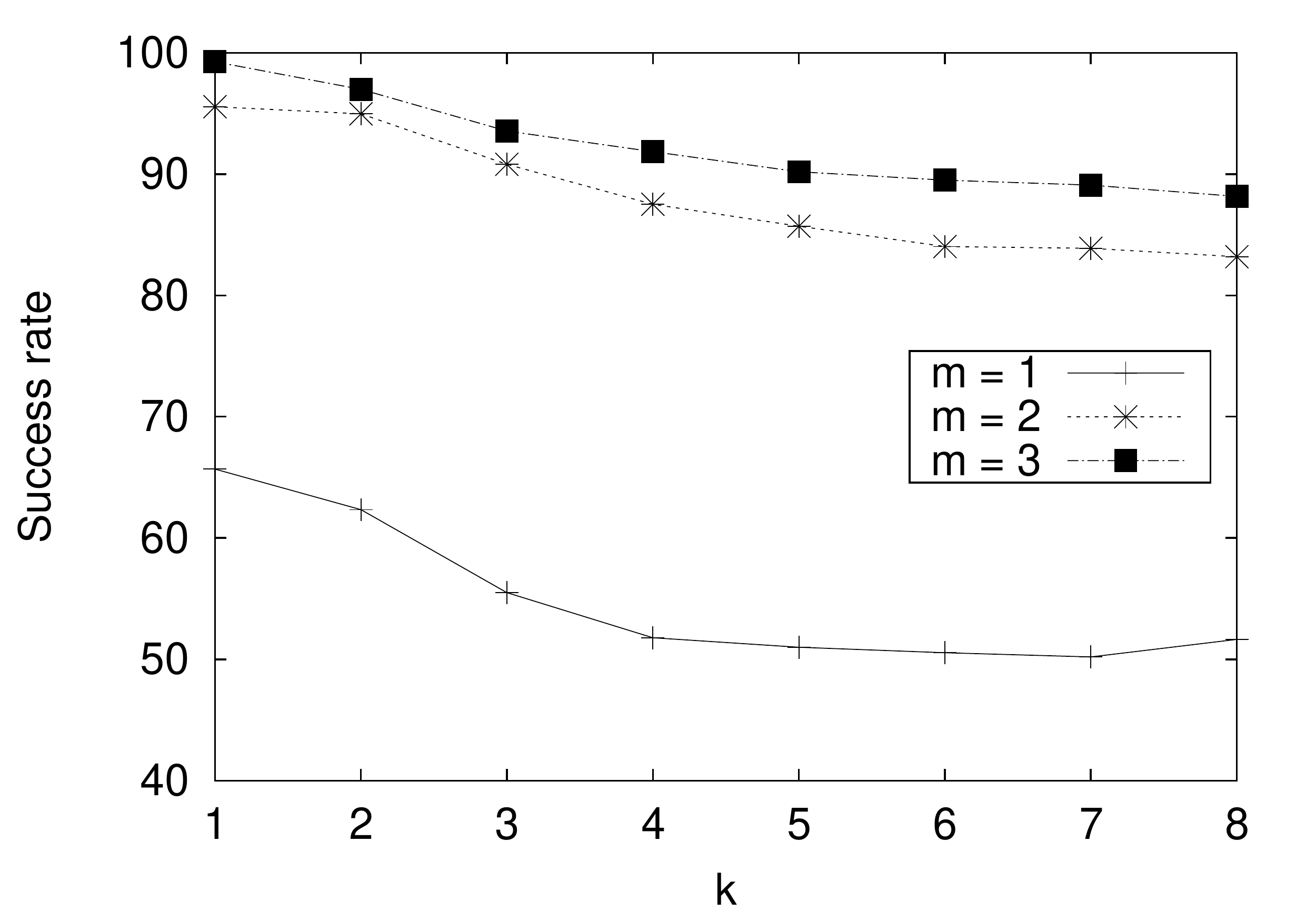}
  	\label{figure:set}
  }
\hspace{-15pt}
  \subfigure[$k$-antiresolving basis]{
       \includegraphics[width=0.35\textwidth,
       angle=270]{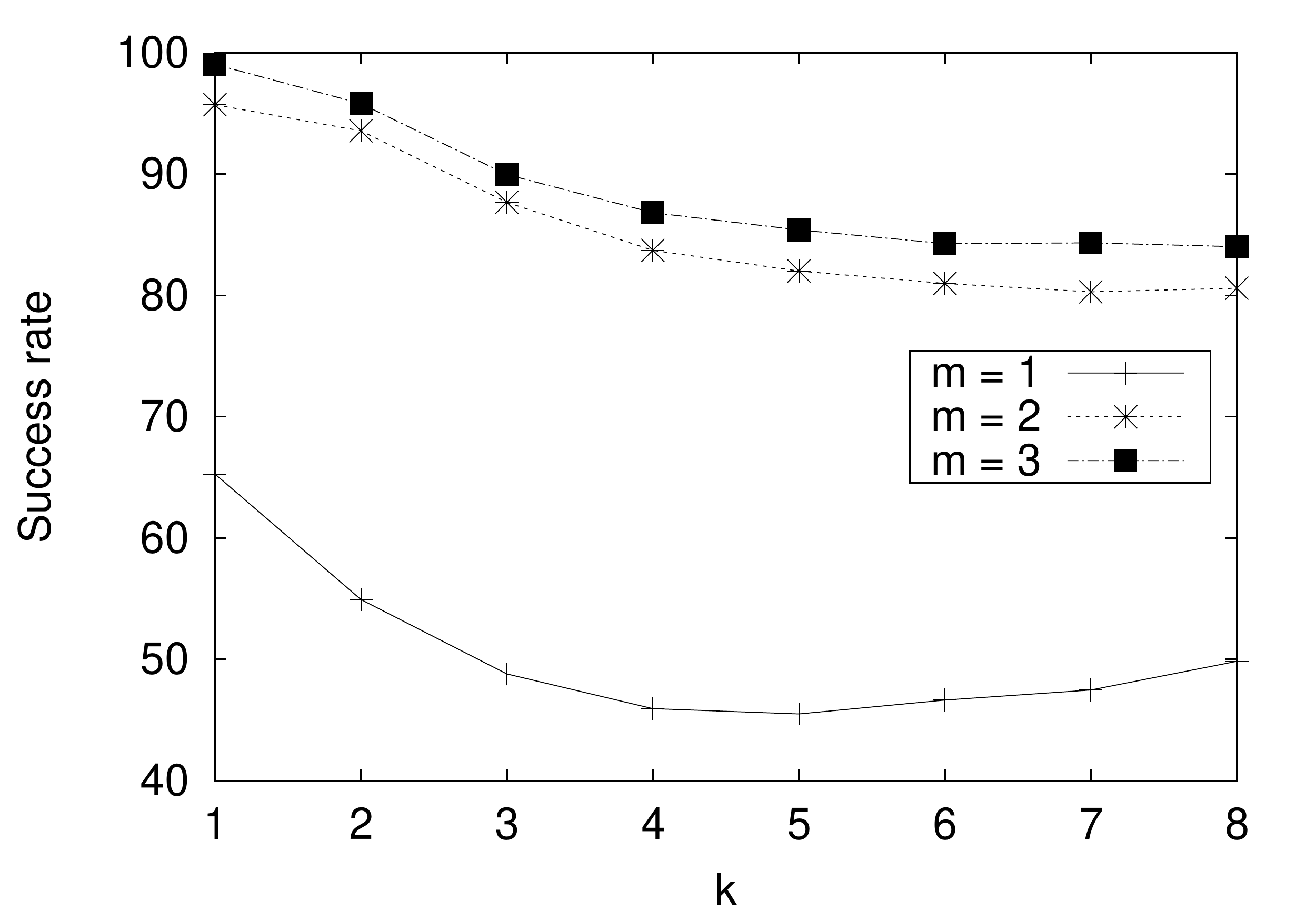}
  	\label{figure:basis}
  }
%  \vspace{-15pt}
  \caption{Figure~\ref{figure:set} and  Figure~\ref{figure:basis} depicts the
  success rates of Algorithm~\ref{alg_antiresolving} and
  Algorithm~\ref{alg_antiresolving_basis}, respectively. The considered values
  of $m$ are
  $\{1,2,3\}$, while $k$ varies from $1$ to $8$.\label{fig:rates}}
%\vspace{-10pt}
\end{figure*}

In the previous section we provided a theoretical impossibility result whereby
a graph can be proven to not contain a $k$-antiresolving set. In
Figure~\ref{fig:ratio}, we
show that such impossibility result can be achieved by
random graphs; with small probability though. It seems also that the percentage
of negative results
monotonically increases with $k$. Indeed, it is easy to prove that this
percentage reaches
its minimum ($0 \%$) and maximum ($100 \%$) when $k$ takes its minimum ($k =
1$) and maximum ($k = N$) value, respectively. However, proving the
monotonicity of the percentage of negative results with respect to $k$ looks
challenging and cumbersome. Figure~\ref{fig:ratio} also shows that the increase
of the success rate of both algorithms when $m$ grows is due to an increase on
both the number of positive and negative results.

\begin{figure*}[htbp]
\centering
  \subfigure[t][$k$-antiresolving set ($m = 1$)]{
     \includegraphics[width=0.33\textwidth, angle = 270]{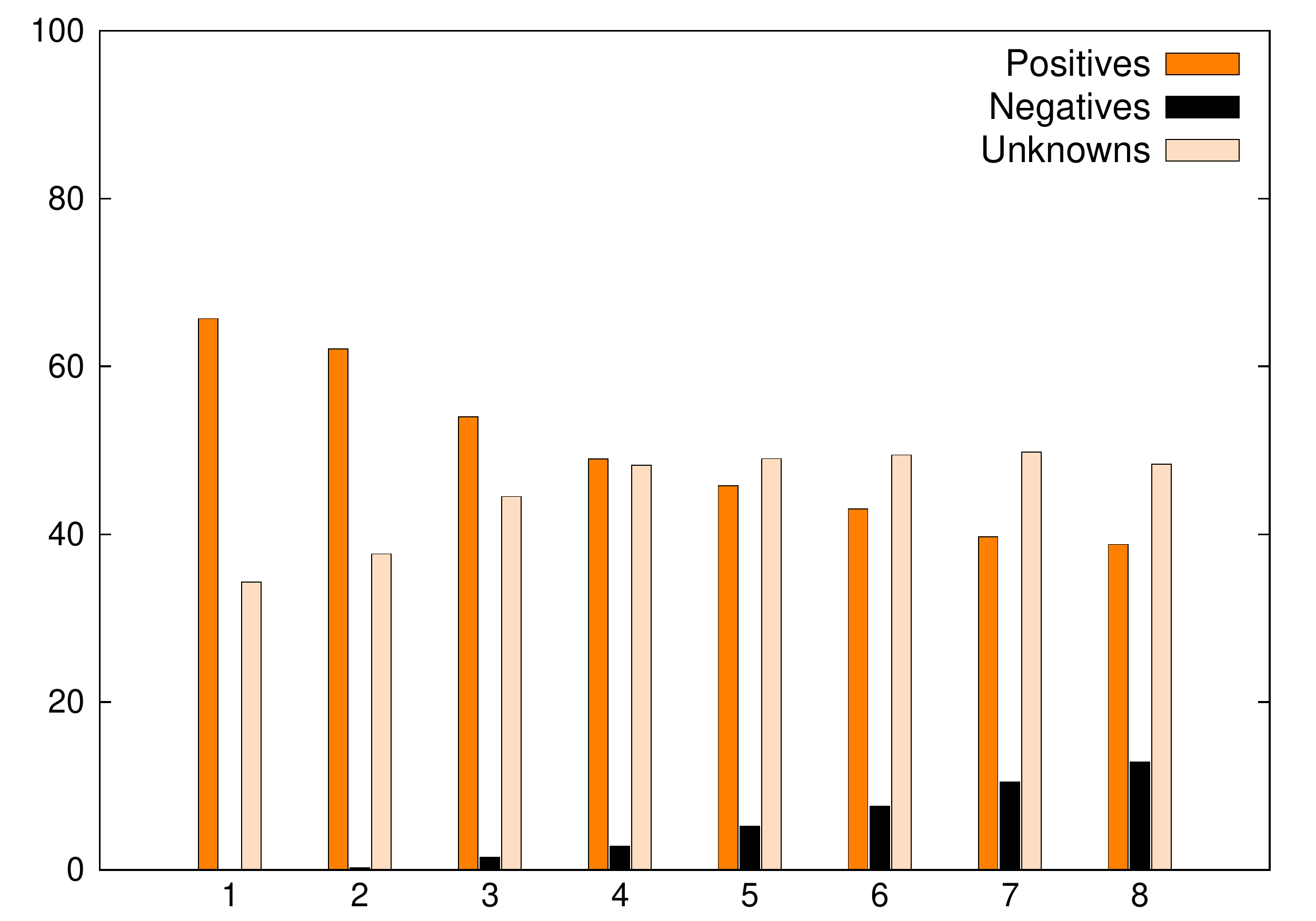}
	\label{figure:set_m_1}
  }
  \subfigure[t][$k$-antiresolving basis ($m = 1$)]{
     \includegraphics[width=0.33\textwidth, angle = 270]{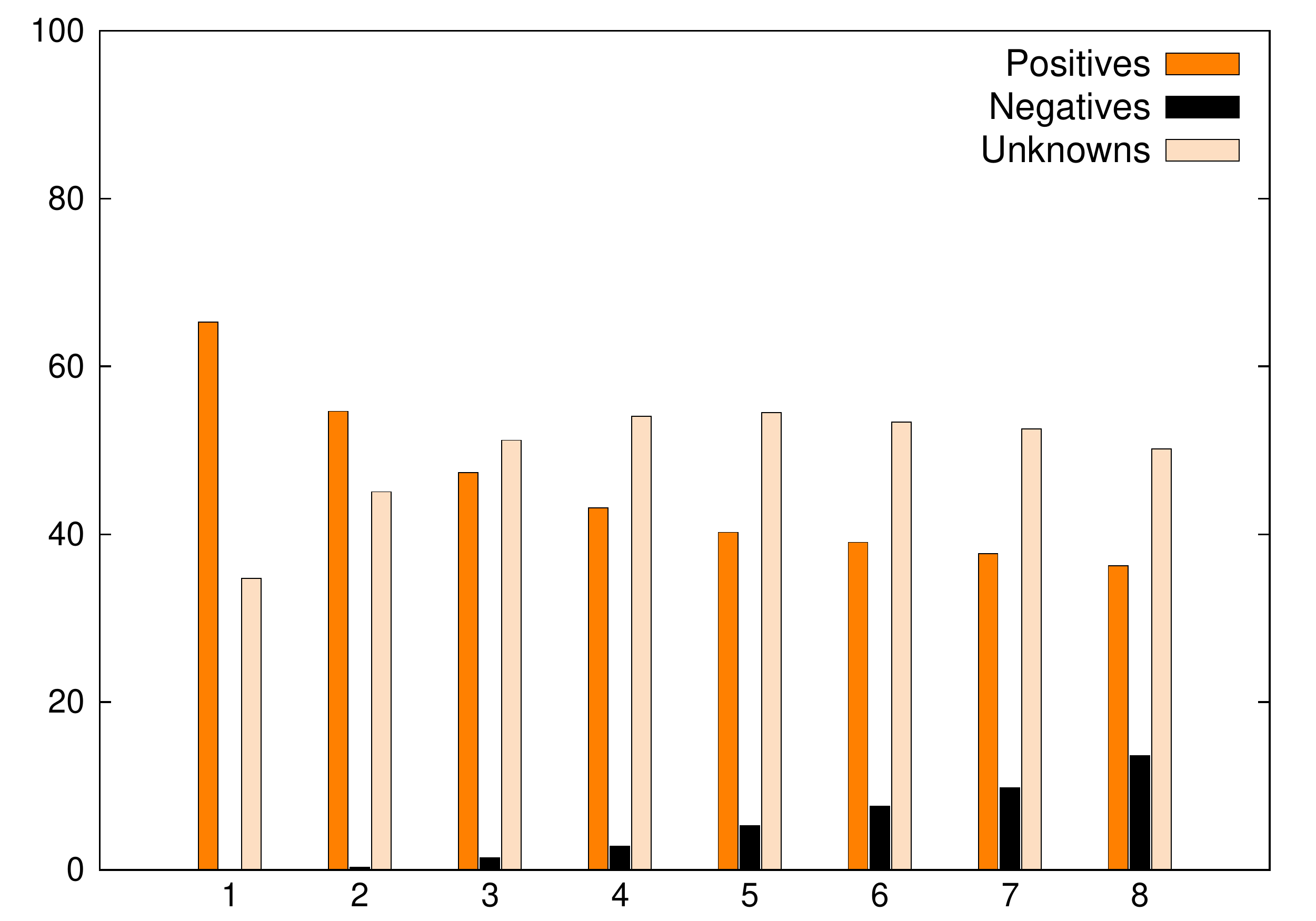}
  	\label{figure:basis_m_1}
  }
  \subfigure[t][$k$-antiresolving set ($m = 2$)]{
     \includegraphics[width=0.33\textwidth, angle = 270]{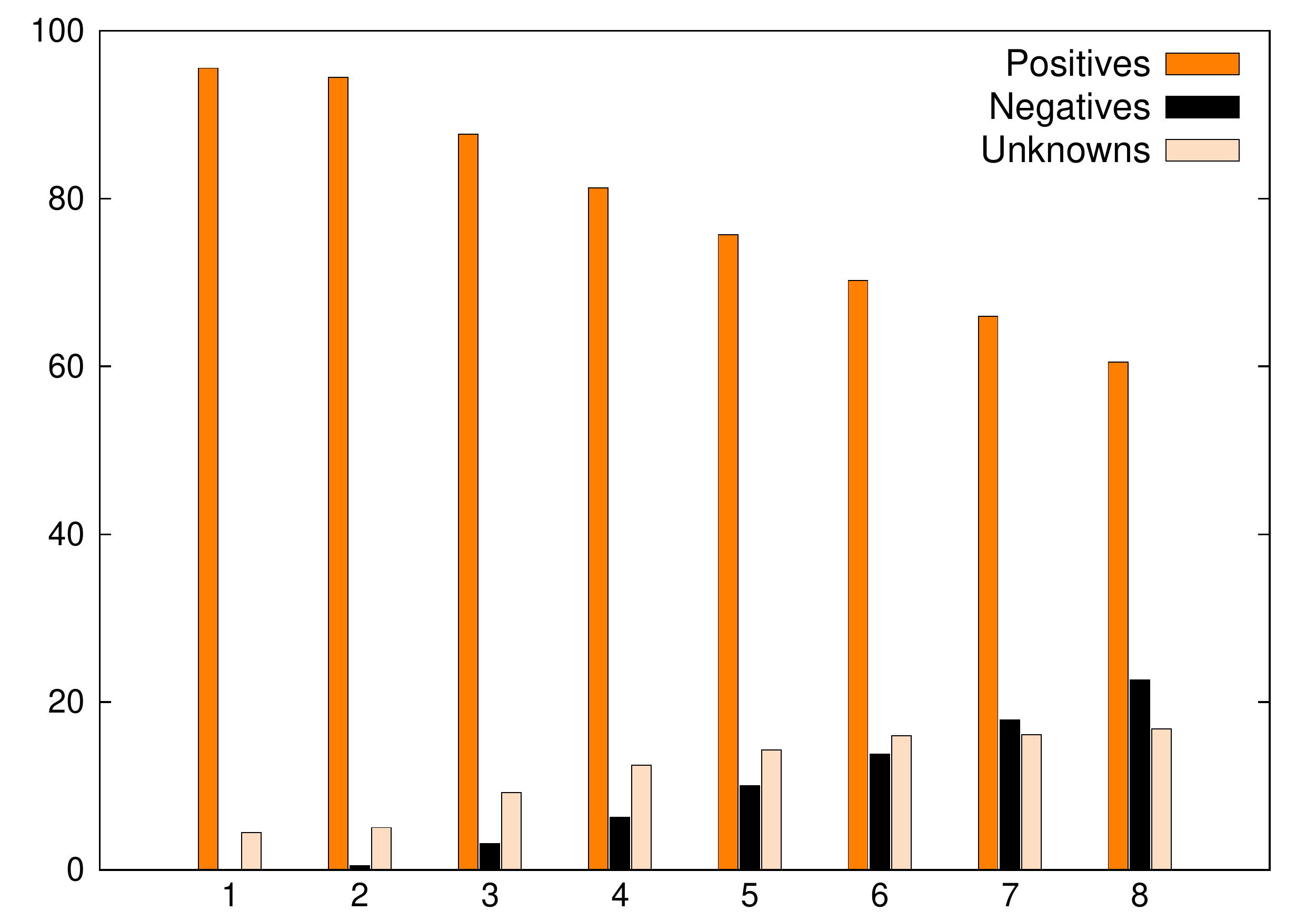}
	\label{figure:set_m_2}
  }
  \subfigure[t][$k$-antiresolving basis ($m = 2$)]{
     \includegraphics[width=0.33\textwidth, angle = 270]{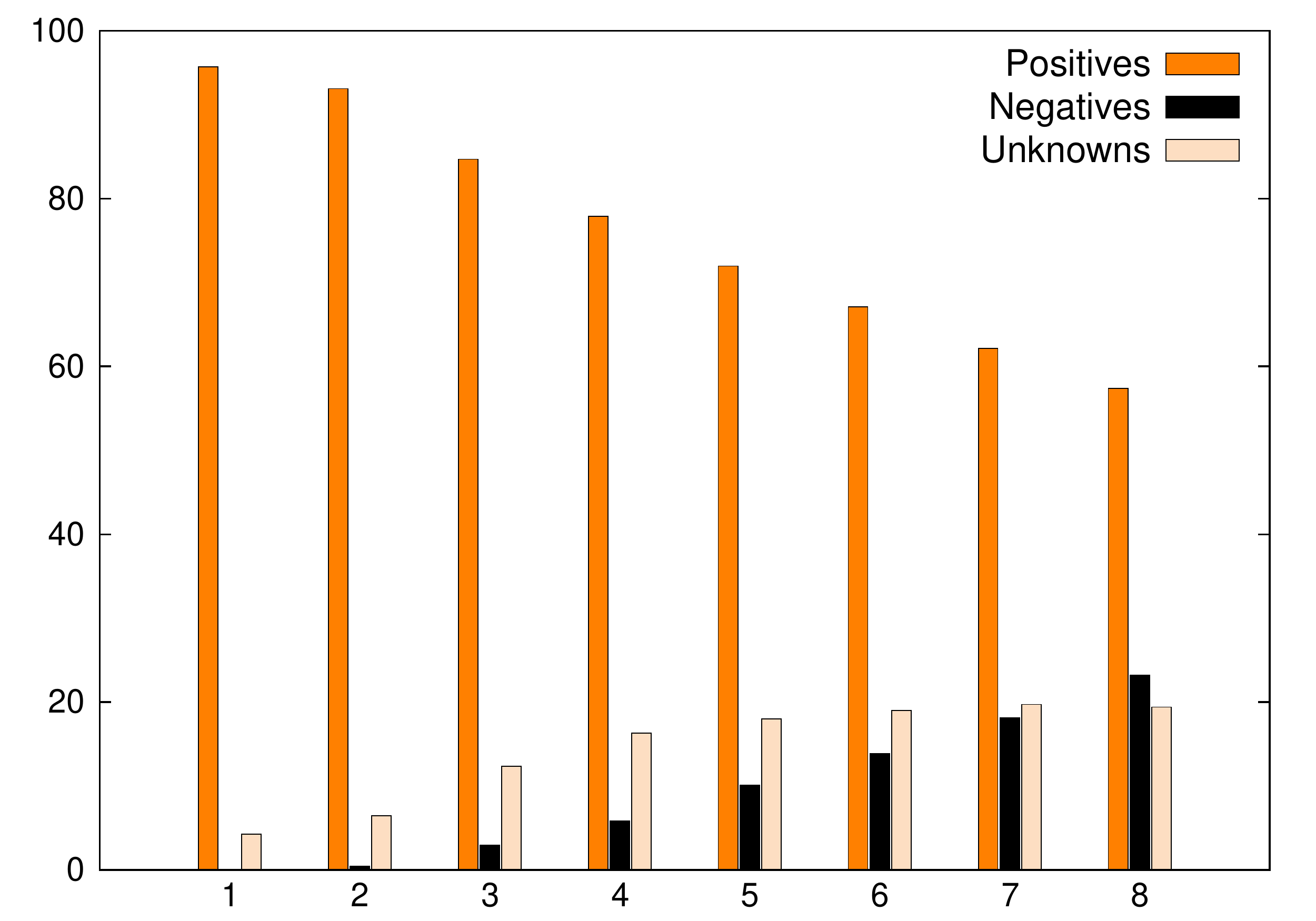}
  	\label{figure:basis_m_2}
  }
  \subfigure[t][$k$-antiresolving set ($m = 3$)]{
     \includegraphics[width=0.33\textwidth, angle = 270]{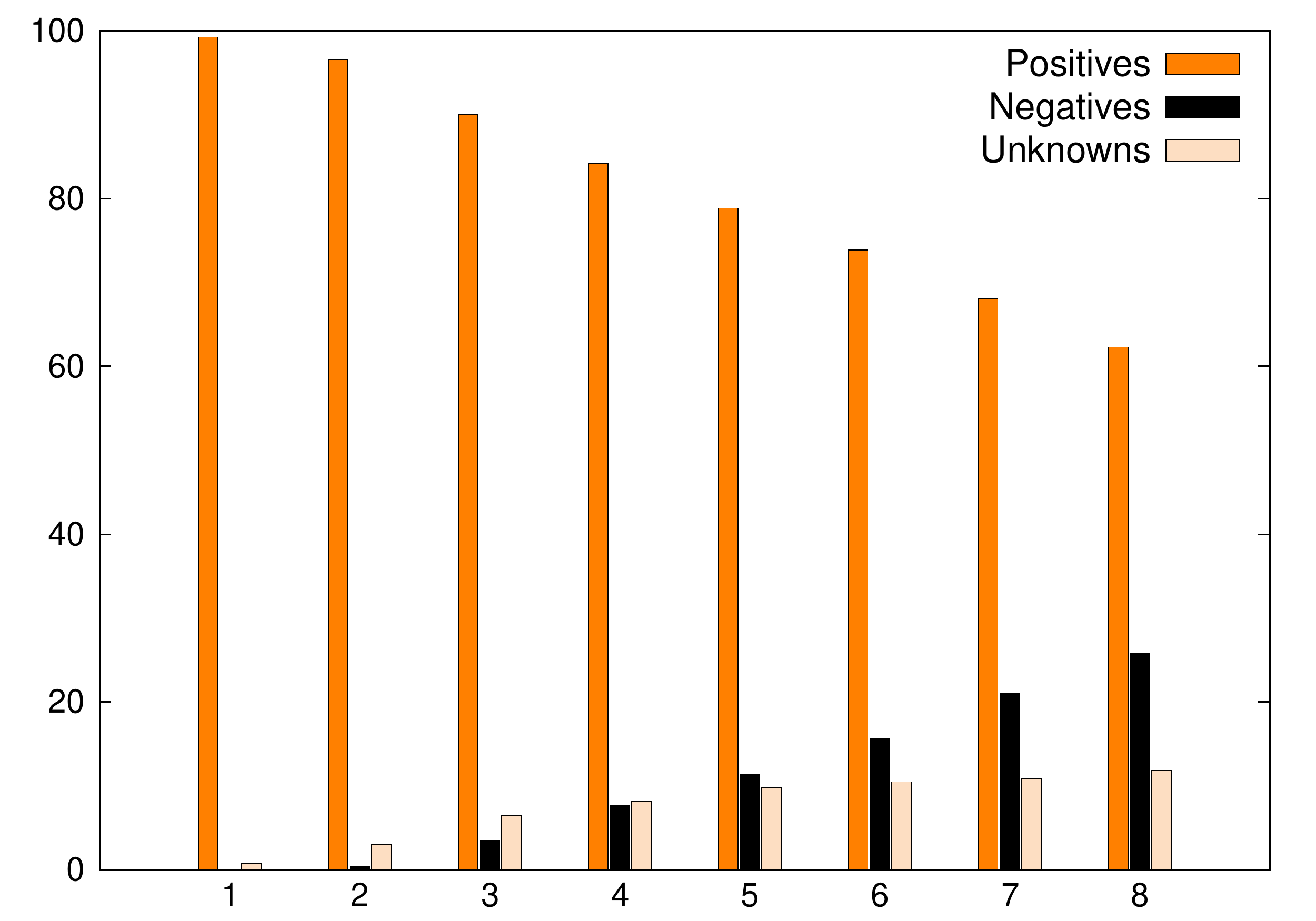}
	\label{figure:set_m_3}
  }
  \subfigure[t][$k$-antiresolving basis ($m = 3$)]{
     \includegraphics[width=0.33\textwidth, angle = 270]{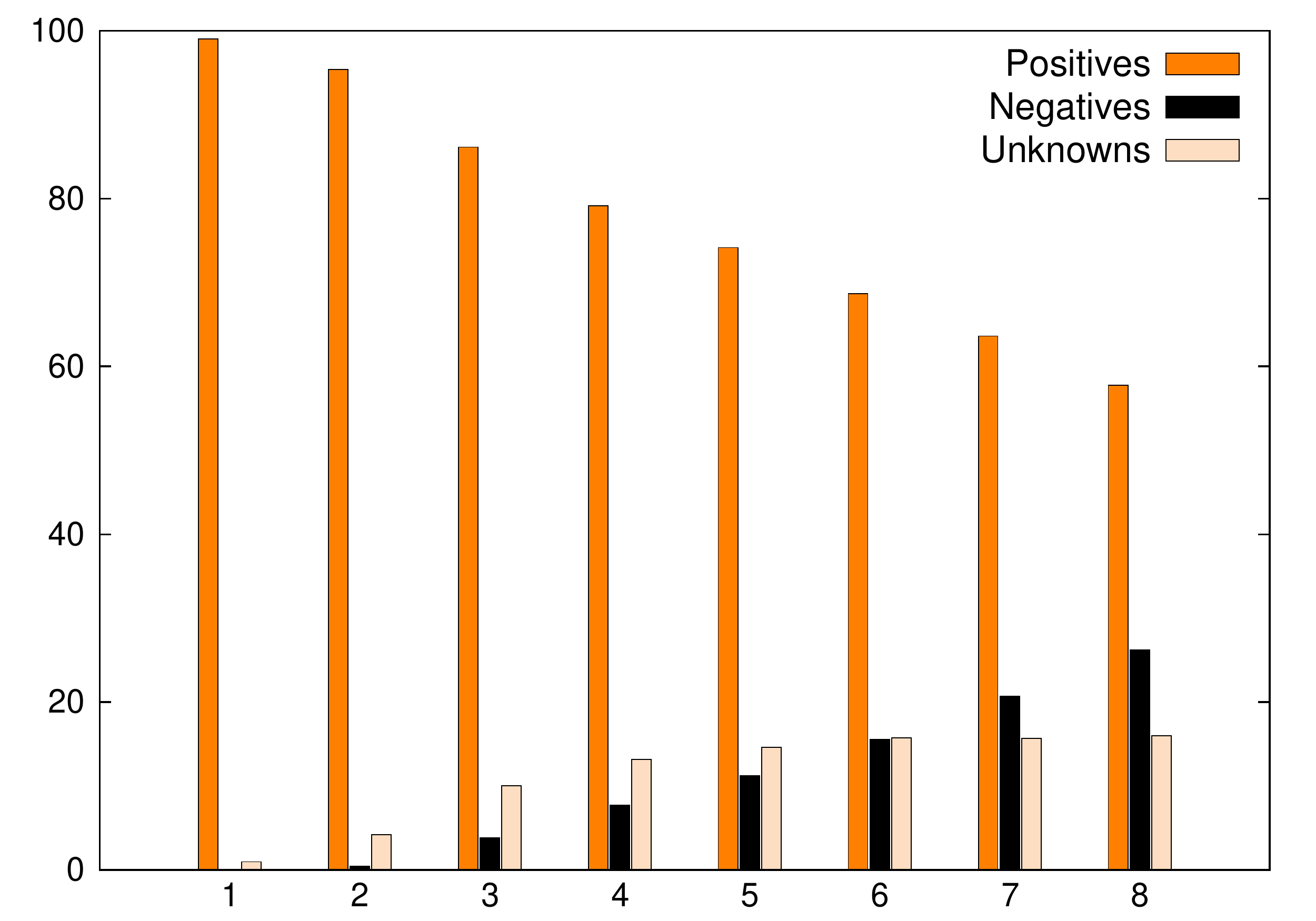}
  	\label{figure:basis_m_3}
  }
  \caption{Six charts showing the ratio of \texttt{true}, \texttt{false},
  and \texttt{unknown} results provided by Algorithm~\ref{alg_antiresolving}
  and Algorithm~\ref{alg_antiresolving_basis} on
  different values for $m \in \{1, 2, 3\}$. Charts at the left are
  devoted to the algorithm aimed at finding
  a $k$-antiresolving set, charts at the right consider the algorithm that
  looks
  for a  $k$-antiresolving basis.\label{fig:ratio}
}
\end{figure*}

\subsection{Empirical evaluation on real-life social graphs}

This section ends with the evaluation of two real-life social graphs with
respect to the proposed privacy measure. The first graph, named
\emph{Facebook graph} in what follows, consists of $10$ ego-networks from
Facebook~\cite{ML2009}. It
contains $4039$ users, $88234$ edges, and $193$ circles. The second graph
describes an online community of students at the University of
California~\cite{POC2009}. In
total, $1899$ students were registered in the network and $13838$ links were
created. We refer to this graph as \emph{Panzarasa graph}.

Both graphs have been analyzed in order to determine the values of $k$ and
$\ell$ such that they $(k, \ell)$-anonymity. Taking into account
the previously presented empirical results on synthetic data, we used for the
analyses Algorithm~\ref{alg_antiresolving_basis} with $m = 2$ as a good
trade-off between performance and success rate. The results are as follows.

The Panzarasa graph contains a
	$1$-antiresolving basis of size $1$. This
	means that this graph does not satisfy $(k,\ell)$-anonymity for $k > 1$
	unless 	$\ell < 1$, which is a meaningless scenario. Similarly, the
	$k$-metric
	antidimension of the Facebook graph is $1$
	for $k = 1$. Hence, it satisfies $(1,1)$-anonymity only;
	the lowest privacy guarantee with respect to our privacy measure.

Our results show that neither the Panzarasa nor the Facebook graph provide
privacy guarantees against active attack. This is not surprising since these
graphs have not been anonymized to prevent any type of structural
attack. Future work thus should be oriented to anonymization methods that
consider $(k,\ell)$-anonymity as a privacy goal.

\section{Mathematical properties on the $k$-metric
antidimension of graphs}\label{sec_metric_antidimension}

In the next two sections we provide some primary theoretical
results on the $k$-metric
antidimension problem. We focus on giving
mathematical properties that, supported by the results in
Section~\ref{sect-algo}, can determine or bound the $k$-metric antidimension
for some families of graphs. In particular in this section, we study the
$k$-metric antidimension of cycles, paths, complete bipartite graphs,  and
other graph families satisfying some specific conditions. To do so, we first
observe some basic properties of $k$-antiresolving sets, which will be further
used.

\begin{observation}\label{basic-prop}$\,$
\begin{enumerate}[{\rm (i)}]
	\item Any resolving set is also a $1$-antiresolving set.
	\item There does not exist $k > 1$ such that all the vertices of a graph
	form a $k$-antiresolving set.
    \item There does not exist any $n$-antiresolving set in a graph of order
    $n$.
	\item Not for every graph $G$ of order $n$ and every $k < n$, there
	exist a $k$-antiresolving set in $G$. For instance, if $G$ is a path
	graph, for $k\ge 3$ there does not exist a
	$k$-antiresolving set in $G$.	
\end{enumerate}
\end{observation}

In order to continue with our study we need to introduce some terminology and
notation. For a
graph $G$ and a vertex $v\in V(G)$, the set $N_G(v)=\{u\in V:\; uv\in E(G)\}$ is
the \emph{open neighborhood} of $v$ and the set $N_G[v] = N_G(v)\cup \{v\}$ is
the \emph{closed neighborhood} of $v$. Two vertices $x$, $y$ are called
(\emph{false}) \emph{true twins} if ($N_G(x) = N_G(y)$) $N_G[x] = N_G[y]$. In
this sense, a vertex $x$ is a \emph{twin} if there exists $y\ne x$ such that
they are either true or false twins. The diameter of $G$ is defined as
$D(G)=\max_{u,v\in V}\{d_G(u,v)\}$.

As mentioned before, not for every integer $k$ it is possible to find a
$k$-antiresolving set in a graph $G$. Thus, it is desirable to analyze first
the interval of suitable values for $k$ satisfying that $G$ contains at least
one $k$-antiresolving set. According to Definition \ref{def_antiresolving_set}
we present the following concept, which is relevant in the study of the
$k$-metric antidimension of graphs.

\begin{definition}[$k$-metric antidimensional graph]
A simple connected graph $G = (V, E)$ is $k$-\emph{metric antidimensional}, if
$k$ is the largest integer such that $G$ contains a $k$-antiresolving set.
\end{definition}

\subsection{$k$-metric antidimensional graphs}\label{subsect-met-antidim}

In order to study the $k$-metric antidimension of graphs, we first focus into
obtaining the values of
$k$ for which a given graph is $k$-metric antidimensional. First notice
that any graph $G$ is always $k$-metric antidimensional for some $k\ge 1$, and
a natural upper bound for $k$ which makes that $G$ would be $k$-metric
antidimensional is clearly the maximum degree of the graph, since the number of
vertices at distance one from any vertex is at most the maximum degree of the
graph.

\begin{observation}\label{remark-Delta}
If $G$ is a connected $k$-metric antidimensional graph of maximum degree
$\Delta$, then $1\le k\le \Delta$.
\end{observation}

Since the maximum degree of a graph is at most the order of the graph minus
one, a particular case of the above result is the next one.

\begin{remark}
If $G$ is any connected $k$-metric antidimensional graph of order $n$, then
$1\le k\le n-1$.
Moreover, $G$ is $(n-1)$-metric antidimensional if and only if $G$ has maximum
degree $n-1$.
\end{remark}

\begin{proof}
The upper bound is a particular case of Remark \ref{remark-Delta}. Now, it is
straightforward to observe that if $v$ is a vertex of $G$ of degree
$n-1$, then for every vertex $u,w\ne v$ it follows that
$r(u|\{v\})=r(w|\{v\})=1$, \emph{i.e.}, every vertex
different from $v$ has the
same metric representation with respect to $\{v\}$. Thus, $\{v\}$ is a
$(n-1)$-antiresolving set, since there exists no $n$-antiresolving sets in $G$.
Thus, $G$ is $(n-1)$-metric antidimensional.

On the contrary, we assume that $G$ is $(n-1)$-metric antidimensional. Hence,
if $S$ is a $(n-1)$-antiresolving set, then $|S|=1$. Thus, the only possibility
is that $S$ is formed by a single vertex and that every vertex is adjacent to
it.
\end{proof}

%To continue with our study we need some extra notation. The
%\emph{eccentricity}
%$\epsilon(v)$ of a vertex $v$ in a connected graph $G$
%is the maximum graph distance between $v$ and any other vertex $u$ of $G$.
%Given a vertex $u$ of a graph $G$, we consider the following local parameter.
%For every $i\in \{1,...,\epsilon(u)\}$, let $d_i(u)=\{v\in
%V(G)\,:\,d(v,u)=i\}$. Now, for every $u\in V(G)$, let
%$$\phi(u)=\min_{1\le i\le \epsilon(u)}\{|d_i(u)|\}.$$
%\label{phi-page}
%and for any graph $G$, let $\phi(G)=\max_{v\in V(G)}\{\phi(v)\}$.

To continue with our study we need some extra
notation. The \emph{eccentricity}
$\epsilon(v)$ of a vertex $v$ in a connected graph $G$
is the maximum length of a shortest path between $v$ and any other vertex $u$
of $G$. Notice that the maximum of the eccentricities of any vertex of $G$ is
the diameter of $G$ and the minimum of the eccentricities is the radius of $G$.
For more information on vertex eccentricity in graphs, see for instance
\cite{west-book}. Figure \ref{example-ecc} shows a graph $G$ and a table with
the eccentricities of all its vertices.
Given a vertex $u$ of a graph $G$, we consider the following local parameter.
For every $i\in \{1,...,\epsilon(u)\}$, let $d_i(u)=\{v\in
V(G)\,:\,d(v,u)=i\}$. Now, for every $u\in V(G)$, let
$$\phi(u)=\min_{1\le i\le \epsilon(u)}\{|d_i(u)|\}.$$
\label{phi-page}
and for any graph $G$, let $\phi(G)=\max_{v\in V(G)}\{\phi(v)\}$. The table and
the graph of Figure \ref{example-ecc} clarify the notation above.

\begin{figure}[ht]
\centering
\includegraphics[width=0.14\textwidth]{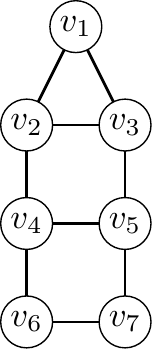}
\hspace*{1cm}
\includegraphics[width=0.7\textwidth]{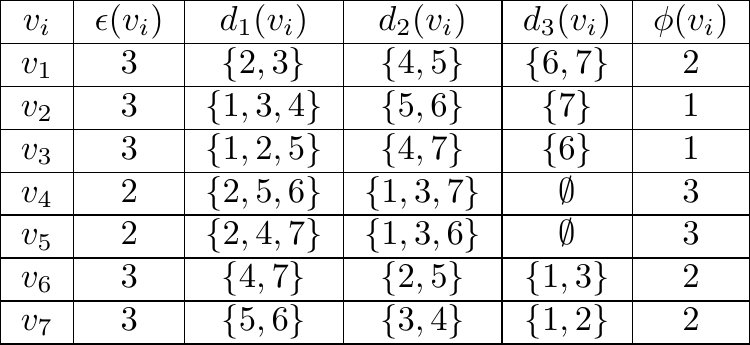}
\caption{Eccentricities of vertices of a graph $G$ and a table which shows that
$\phi(G)=3$.}
\label{example-ecc}
\end{figure}

\begin{theorem}\label{dimensional-phi}
Any connected graph $G$ is $k$-metric antidimensional for some $k\ge \phi(G)$.
\end{theorem}

\begin{proof}
Assume $x$ is a vertex of degree at least two in $G$ such that
$\phi(G)=\phi(x)$. Thus, for any vertex $y\ne x$, there exist at least
$\phi(x)-1$ vertices $v_1,v_2,....v_{\phi(G)-1}$ in $V(G)-\{x,y\}$ such that
$d(y,x)=d(v_1,x)=...=d(v_{\phi(T)-1},x)$. Moreover, since $\phi(G)=\phi(x)$,
there exists at least one vertex $y'$ such that there are exactly $\phi(G)-1$
different vertices satisfying the above mentioned. So, $\{x\}$ is a
$\phi(x)$-antiresolving set and $G$
is $k$-metric antidimensional for some $k\ge \phi(G)$.
\end{proof}

%If a graph $G$ is $1$-metric antidimensional, then its $1$-antiresolving sets
%are standard resolving sets as defined in \cite{Harary1976,Slater1975} and this
%has been studied in several scientific articles
%(\emph{e.g.,} \cite{Bailey2011a, Bailey2011, Chartrand2000, Haynes2006,
%Yero2011}). According to that, in this work we are interested mainly in those
%graphs
%being $k$-metric antidimensional for some $k\ge 2$. An example of a graph being
%$1$-metric antidimensional is for instance the path graph of even order.
%A interesting question regarding this could be the following one. Is it
%possible to characterize all graphs being $1$-metric antidimensional?

Now, notice that if a graph $G$ is a $1$-metric antidimensional, then every subset of vertices $S$ ought to be a $1$-antiresolving set, implying that $adim(G) = 1$. According to that fact, in this work we are mainly interested in those
graphs being $k$-metric antidimensional for some $k\ge 2$. An example of a
graph being $1$-metric antidimensional is for instance the path graph of even order.

\subsection{Graphs that are $k$-metric antidimensional for some $k\ge 2$}

To begin with the description of some families of graphs being
$k$-metric antidimensional for some $k\ge 2$ we define the \emph{radius} and the
\emph{center} of a graph as follows.
The \emph{radius} $r(G)$ of $G$ is the minimum eccentricity of any vertex in
$G$.
The \emph{center} of $G$ is the set $S$ of vertices of $G$ having eccentricity
equal to the radius of $G$.

\begin{remark}
If the center of a graph $G$ is only one vertex, then $G$ is $k$-metric
antidimensional for some $k\ge 2$.
\end{remark}

\begin{proof}
Let $v$ be the center of $G$. Hence, there exist two diametral
vertices $u,w$ such that $d_G(v,u)=d_G(v,w)=\epsilon(v)=r(G)$. Since $v$ has
eccentricity $r(G)$, there is no vertex $z\ne u,w$ in $G$ such that
$d_G(v,z)>d_G(v,u)=d_G(v,w)$. Thus, it follows that for
any vertex $x\ne v$ there exists at least a vertex $y$ belonging to the $u-w$
path such that
$d_G(x,v)=d_G(y,v)$. Therefore, $\{v\}$ is a $k$-antiresolving set in $G$ for
some $k\ge 2$.
\end{proof}

If a path graph has odd order, then its center is formed by only one vertex.
Also, for every vertex of any path, there exists at most other different
vertex having equal distance to a third vertex of the path. Thus, it is
clear the following consequence of the Remark above.

\begin{corollary}\label{paths}
If a path $P_n$ has odd order, then it is $2$-metric antidimensional.
\end{corollary}

Another example of $2$-metric antidimensional are the cycle graphs as we next
see.

\begin{remark}\label{cycles}
Any cycle graph $C_n$ is $2$-metric antidimensional.
\end{remark}

\begin{proof}
We assume first that $n$ is odd. Let $v$ be any vertex of $C_n$. Hence, for any
vertex $x\ne v$ of $C_n$, there exists only one $y\ne x,v$ such that
$d_{C_n}(x,v)=d_{C_n}(y,v)$. Thus, $\{v\}$ is a $2$-antiresolving set in $C_n$.
Assume now that $n$ is even and let $\{u,w\}$ be any two diametral vertices of
$C_n$. We observe that for any vertex $x\ne u,w$ of
$C_n$, there exists only one $y\ne x,u,w$ such that $d_{C_n}(x,u)=d_{C_n}(y,u)$
and
$d_{C_n}(x,w)=d_{C_n}(y,w)$. Thus, $\{u,w\}$ is a $2$-antiresolving set in
$C_n$.

On the other hand, there does not exists $k>2$ such that $C_n$ contains a
$k$-metric
antiresolving set, since for any vertex of $C_n$, there exists at most another
different vertex having equal distance to a third vertex of the path. Therefore,
$C_n$ is $2$-metric antidimensional.
\end{proof}

If the vertices of a set $S$ are pairwise twins in a graph $G$, then it is
clear that they have the same distance to every other vertex $x\notin S$.
So, $V(G)-S$ is a $|S|$-antiresolving set for $G$. Hence, the following result.

\begin{observation}\label{obs:twin-vertices}
If the vertices of a set $S$ are pairwise twins in a graph $G$, then $G$
is $k$-metric antidimensional for some $k\ge |S|$.
\end{observation}

Complete bipartite graph\footnote{A graph $G$ is complete bipartite if its
vertex set can be divided into two disjoint sets $U$ and $V$ such that every
vertex in $U$ is adjacent to every vertex in $V$ and no more.} are special kind
of graphs, since they have a bipartition of the vertex set in which all the
vertices belonging to one of the sets of the bipartition are pairwise twin
vertices.
Let $K_{r,t}$ be a complete bipartite graph. Next we analyze the suitable
values $k$ making a complete bipartite graph $k$-metric antidimensional.

\begin{remark}\label{bipartite}
Any complete bipartite graph $K_{r,t}$ with $r\ge t$ is $r$-metric
antidimensional.
\end{remark}

\begin{proof}
Let $U$ and $V$ be the two disjoint sets of $K_{r,t}$ with $|U|=r$ and $|V|=t$.
Notice that $U$ (respectively $V$) is a set of pairwise twin vertices. Thus, by
Observation \ref{obs:twin-vertices} we have that $K_{r,t}$ is $k$-metric
antidimensional
for some $k\ge |U|=r$. Suppose that $k\ge r+1$ and let $S$ be a
$k$-antiresolving set for $K_{r,t}$. Since every vertex of $K_{r,t}$ has
distance either one or two to any other vertex of $K_{r,t}$ it is not possible
to find $k-1$ vertices having the same distance to every vertex of $S$, a
contradiction. So, $k=r$ and the proof is complete.
\end{proof}

\subsection{The $k$-metric antidimension of
graphs}\label{subsect-metric-antidim}

In this subsection we compute the $k'$-metric antidimension of some graphs which
were already described to be $k$-metric antidimensional for some value $k\ge
k'$. It is clear that the first natural bound which follows for the $k$-metric
antidimension of a graph of order $n$ is $\adim_k(G)\le n-k$. Such a bound is
tight. It is achieved, for instance, for the complete bipartite graphs
$K_{r,t}$ as we can see at next by taking the case $t< k\le r$.

\begin{proposition}
Let $r,t$ be two positive integers with $r\ge t$.
\begin{enumerate}
\item If $t< k\le r$, then $\adim_k(K_{r,t})=r+t-k$.
\item If $1< k\le t$, then $\adim_k(K_{r,t})=r+t-2k$.
\end{enumerate}
\end{proposition}

\begin{proof}
From Proposition \ref{bipartite} we know that $K_{r,t}$ is a $r$-metric
antidimensional graph. Let $U$ and $V$ be the two partite sets of $K_{r,t}$
with $|U|=r$ and $|V|=t$. We assume first that $t< k\le r$. Let $A\subseteq U$
with $|A|=k$ and let be $S=(V\cup U)-A$. Notice that if $k=r$, then $A=U$ and
so, $S=V$. Since any vertex $v\notin S$ (or equivalently $v\in A$) is adjacent
to every vertex of $V$ and it has distance two to every vertex in $U-A$, we
have that all the vertices of $A$ have the same metric representation with
respect to $S$. As $|A|=k$, it follows that $S$ is a $k$-antiresolving set and
$\adim_k(K_{r,t})\le r+t-k$. Now, suppose $\adim_k(K_{r,t})<r+t-k$ and let $S'$
be a
$k$-antiresolving set for $K_{r,t}$. So, we have either one of the following
situations.
\begin{itemize}
\item There exist more than $k$ vertices of $U$ not in $S'$. Hence, for any
vertex $u\in U-S'$ there exist at least $k$ vertices not in $S'$ which,
together with $u$, have the same metric representation with respect to $S'$.
So, $S'$ is not a $k$-antiresolving set, but a $k'$-antiresolving set for some
$k'\ge k+1$, a contradiction.
\item There exists at least one vertex of $V$ not in $S'$. It is a direct
contradiction, since $|V|=t<k$.
\end{itemize}
Therefore, we obtain that $\adim_k(K_{r,t})=r+t-k$.

On the other hand, we assume that $1< k\le r$. Let $X\subseteq U$ with $|X|=k$,
let $Y\subseteq V$ with $|Y|=k$ and let $Q=(V-Y)\cup (U-X)$. Hence,
for any vertex $v\notin Q$ (or equivalently $v\in X\cup Y$), there exist
exactly $k-1$ vertices, such that all of them, together with $v$, have the same
metric
representation with respect to $Q$. Thus, $Q$ is a $k$-antiresolving set and
$\adim_k(K_{r,t})\le r+t-2k$.

Now, suppose that $\adim_k(G)<r+t-2k$ and let $Q'$ be a $k$-antiresolving set
in $K_{r,t}$. Hence, either there exist more than $k$ vertices of $U$ not in
$Q'$ or there exist more than $k$ vertices of $V$ not in $Q'$. As above, in any
of both possibilities we obtain that $Q'$ is not $k$-antiresolving set, but a
$k'$-antiresolving set for some $k'\ge k+1$, a contradiction.
As a consequence, we obtain that $\adim_k(K_{r,t})= r+t-2k$.
\end{proof}

Next we study the $k$-metric antidimension of some other families of basic
graphs. According to Remark \ref{cycles} we know that the cycles $C_n$ are
$2$-metric antidimensional and, by Corollary \ref{paths}, that the paths $P_n$
are $2$-metric antidimensional only in the case $n$ is odd. Next we compute its
$2$-metric antidimension.

\begin{proposition}
Let $n\ge 2$ be an integer. Then
$$\adim_2(P_{2n+1})=1\;\;\mbox{and}\;\;\adim_2(C_n)=\left\{\begin{array}{ll}
                                               1, & \mbox{if $n$ is odd,} \\
                                               2, & \mbox{if $n$ is even.}
                                             \end{array}\right.
$$
\end{proposition}

\begin{proof}
If $v$ is the center of a path $P_{2n+1}$, then for any
other vertex $u\ne v$ there exists exactly one vertex $w\ne v,u$ such that
$w,u$ have the same metric representation with respect to
$\{v\}$. Thus $\adim_2(P_{2n+1})=1$.

Suppose $n$ is even and let $u,w$ be two diametral vertices in $C_n$. We
observe that for any vertex $x\ne u,w$ there exists exactly one vertex $y\ne
x,u,w$  in $C_n$, such that $x,y$ have the same metric
representation with respect to $\{u,w\}$. Thus, $\adim_2(C_{2n})\le 2$. To
see that $\adim_2(C_{2n})=2$ we can observe that any set with only one vertex
$h$ is not a $1$-antiresolving set, since for the vertex $f$ being diametral
with $h$ there does not exist any other vertex $f'$ having the same metric
representation with respect to $\{h\}$. On the other hand, if $n$ is odd and
$a$ is any vertex of $C_n$, then we can check that for
any vertex $b\ne a$, there exists exactly one vertex $c\ne a,b$, such
that $b,c$ have the same metric representation with respect to $\{a\}$. Thus,
$\adim_2(C_{2n+1})=1$.
\end{proof}

\section{The particular case of trees}\label{section-trees}

Let $T$ be a tree and let $u$ be a vertex of $T$ of degree at least two. Let
$v$ be a neighbor of $u$. A $v$-branch of $T$ at $u$ is the subtree $T_{u,v}$
obtained from the union of all length maximal paths beginning in $u$, passing
throughout $v$ and finishing at a vertex of degree one in $T$. Given a
$y$-branch $T_{x,y}$ at $x$, we say that $\xi(T_{x,y})$ is the eccentricity of
the vertex $x$ in the $y$-branch $T_{x,y}$.
%the length of the largest path beginning in $x$, passing throughout $y$ and
%finishing at a vertex of degree one in $T_{x,y}$.
Two branches $T_{x,y_1}$ and $T_{x,y_2}$ at $x$ are $\xi_x$-\emph{equivalent}
if $\xi(T_{x,y_1})=\xi(T_{x,y_2})$.
For every vertex $x$ of $T$, let $\xi(x)$
represents the maximum number of pairwise $\xi_x$-equivalent branches at $x$ and
let $l_{\xi}(x)$ equals the length of any $\xi_x$-equivalent branch. Now, for
any tree $T$, we define the following parameter:
$$\xi(T)=\max_{x\in V(T):\delta(x)\ge 2}\{\xi(x)\}.$$
An example which helps to clarify the above definitions is given in Figure
\ref{example}. There we have a tree $T$ satisfying the following. The vertex
$v_5$ has 4 branches: $T_{v_5,v_6}$, $T_{v_5,v_{10}}$, $T_{v_5,v_{15}}$ and
$T_{v_5,v_4}$. For instance
$V(T_{v_5,v_{15}})=\{v_5,v_{15},v_{16},v_{17},v_{11},v_{18},v_{14},v_{13}\}$.
We observe that $\xi(T_{v_5,v_6})=2$, $\xi(T_{v_5,v_{10}})=1$,
$\xi(T_{v_5,v_{15}})=3$ and $\xi(T_{v_5,v_4})=4$. So, $v_5$ has no
$\xi_{v_5}$-equivalent branches and $\xi(v_5)=0$. Similarly, it can be noticed
that $v_3$ and $v_{15}$ are the only vertices of $T$ which have equivalent
branches. That is, $T_{v_3,v_2}$ and $T_{v_3,v_9}$ are $\xi_{v_3}$-equivalent,
since $\xi(T_{v_3,v_2})=\xi(T_{v_3,v_{9}})=2$. Thus $\xi(v_3)=2$ and
$l_{\xi}(v_3)=2$. Analogously,
$\xi(T_{v_{15},v_{11}})=\xi(T_{v_{15},v_{10}})=\xi(T_{v_{15},v_{14}})=2$ and
$\xi(v_{15})=3$, $l_{\xi}(v_{15})=2$. Therefore $\xi(T)=3$.

\begin{figure}[ht]
\centering
\includegraphics[width=0.95\textwidth]{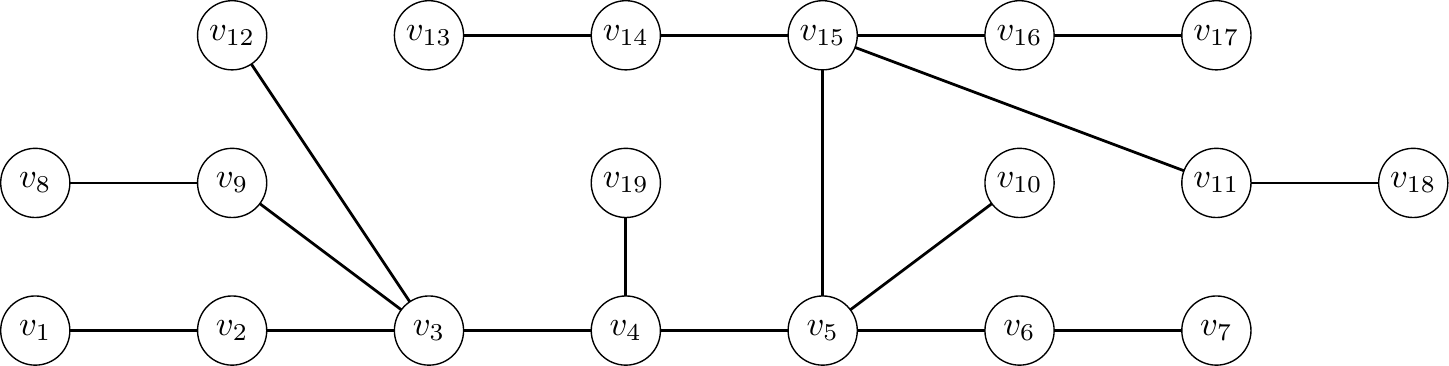}
\caption{A $3$-metric antidimensional tree $T$.}
\label{example}
\end{figure}

Now, for the particular case of
trees, we next use the definition of $\phi(G)$ already presented in Section
\ref{sec_metric_antidimension}. As an example, for the tree of Figure
\ref{example} we have that, for instance, $\phi(v_3)=3$, $\phi(v_4)=3$ and
$\phi(v_5)=2$. Also, some
calculations give that $\phi(T)=3$.

Now, with the definitions above we present the following result.

\begin{theorem}\label{dimensional-trees}
Any tree $T$ is $k$-metric antidimensional for some $k\ge
\max\{\phi(T),\xi(T)\}$.
\end{theorem}

\begin{proof}
From Theorem \ref{dimensional-phi} it follows that $k\ge \phi(T)$. Now, let $x$
be a vertex of degree at least two in $T$ such that $\xi(T)=\xi(x)$. Hence,
there exist $\xi(T)$ disjoint paths beginning in $x$, passing throughout a
vertex $y_j$ (neighbor of $x$), and ending in a vertex $w_j$ of degree one in
$T$ with $j\in \{1,...,\xi(T)\}$. Moreover, every $y_j$-branch $T_{x,y_j}$ does
not contain any other vertex further away from $x$ than $w_j$.

We consider now the set
$$A=V(T)-\left(\bigcup_{i=1}^{\xi(T)}V(T_{x,y_j})\right)\bigcup \{x\}.$$
Notice that for any vertex $u\not\in A$, there exist at least $\xi(T)-1$
different
vertices $v_1,v_2,....v_{\xi(T)-1}$ in $V(T)-A$ such that
$d(u,z)=d(v_1,z)=...=d(v_{\xi(T)-1},z)$ for every $z\in A$. Moreover, since
$\xi(T)=\xi(x)$, it follows that there exists a vertex $u'$ such that there are
exactly $\xi(T)-1$ different vertices satisfying the above mentioned. Thus, $A$
is a $\xi(T)$-antiresolving set and, as a
consequence, $T$ is $k$-metric
antidimensional for some $k\ge \xi(T)$.

Therefore we obtain that $T$ is $k$-metric antidimensional for some
$k\ge\max\{\phi(T),\xi(T)\}$ and the proof is complete.
\end{proof}

According to Theorem \ref{dimensional-trees}, we conclude that the tree $T$ in
Figure \ref{example} is $k$-metric antidimensional for some $k\ge 3$, since
$\phi(T)=3$ and
$\xi(T)=3$; that tree is indeed $3$-antidimensional. If we add some extra
vertices to this mentioned tree, like in Figure \ref{example2}, we obtain that
$\phi(T)=5$ (since $\phi(v_4)=5$), and it remains $\xi(T)=3$. Thus, this new
tree is $k$-metric antidimensional for some $k\ge 5$, and by Remark
\ref{remark-Delta} we have that $k=5$.

\begin{figure}
\centering
     \includegraphics[width=0.95\textwidth]{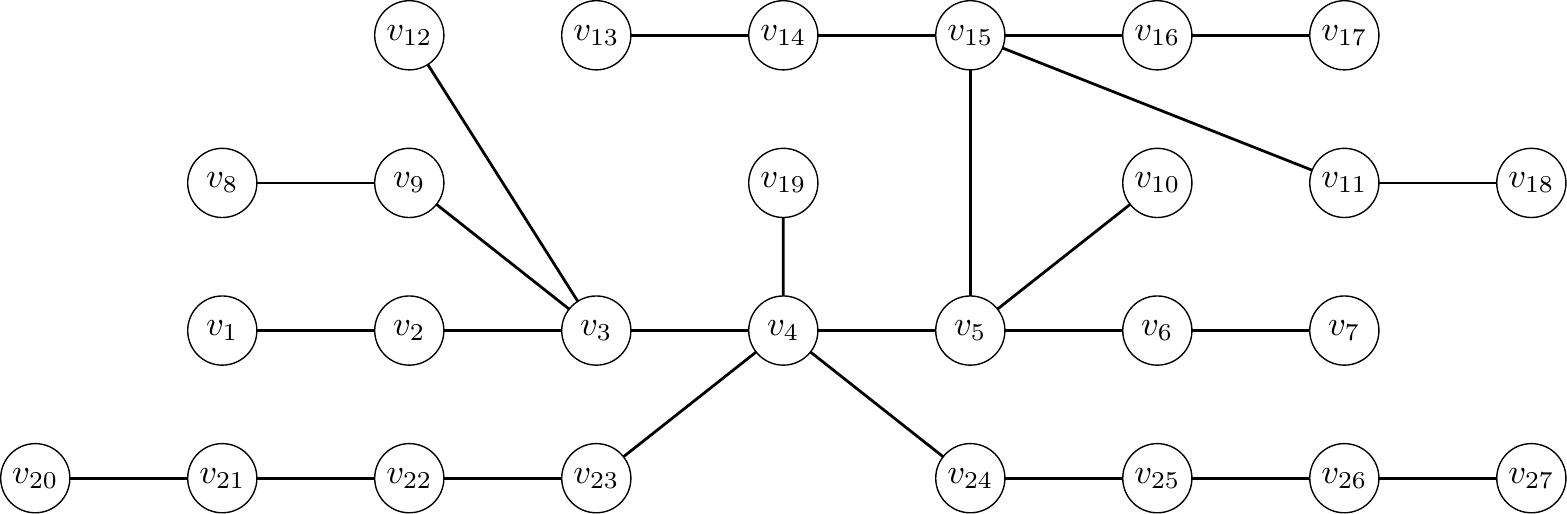}
\caption{A $5$-metric antidimensional tree $T$.}
\label{example2}
\end{figure}

Notice that $\xi(T)>\phi(T)$ holds for some trees. For instance, if we
take a star
graph $S_{1,n}$, $n\ge 4$, and we add an extra vertex $x$ connected by an edge
with one leaf $y$ of $S_{1,n}$, then we have a tree $T$ such that $\phi(T)=2$
(for the vertex $y$, $\phi(y)=2$) and $\xi(T)=n-1$.

Moreover, there are graphs in which the bound of Theorem
\ref{dimensional-trees} is not achieved. An example of this appears in Figure
\ref{example3}. There we have a tree $T$ such that $\xi(T)=3$ ($\xi(v_7)=3$)
and $\phi(T)=3$ ($\phi(v_{12})=3$). Nevertheless the set
$\{v_{10},v_{11},v_{12},v_{13},v_{14},v_{18},v_{19}\}$ is a $4$-antiresolving
set.

\begin{figure}[!ht]
\centering
     \includegraphics[width=0.95\textwidth]{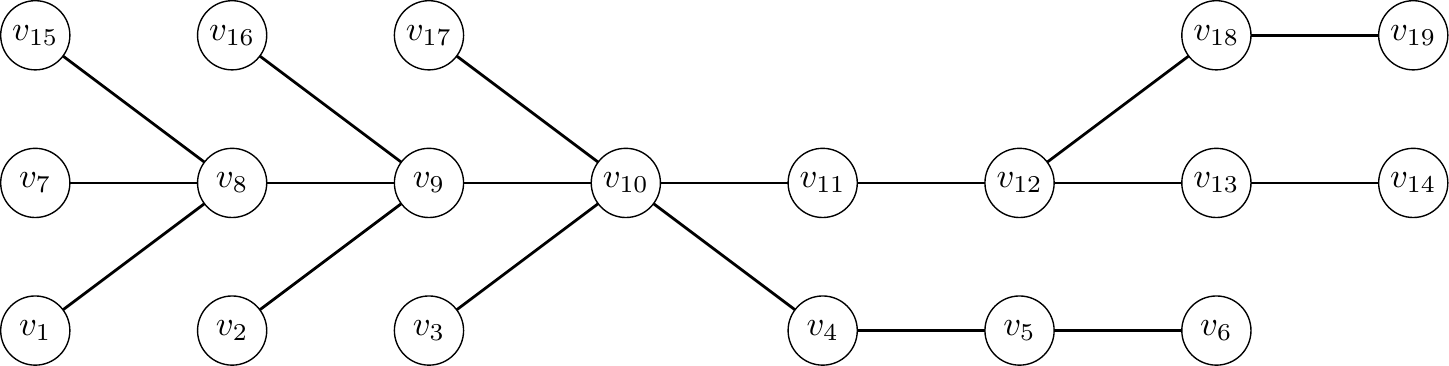}
\caption{The set $\{v_{10},v_{11},v_{12},v_{13},v_{14},v_{18},v_{19}\}$ is a
$4$-metric antiresolving set.}
\label{example3}
\end{figure}

\subsection{The $k$-metric antidimension of trees}

Once we have a lower bound for the integer $k'$ for which a given tree $T$ is
$k'$-metric antidimensional, we are able to compute its $k$-metric
antidimension for a suitable value $k\le k'$. We first notice that if $T$ is
$1$-metric antidimensional, then any $1$-metric antiresolving set is an
standard resolving set, as defined in \cite{Harary1976,Slater1975} and, in such
a case, $\adim_1(T)=dim(T)$. Since it is not our goal to study such a case,
from now on we consider only those trees being $k$-metric antidimensional for
some $k\ge 2$.

\begin{remark}\label{remark-phi-equal-1}
Let $T$ be a $k$-metric antidimensional tree and let $x\in V(T)$ such that
$\phi(x)=t$. If $t\ge 2$, then $\adim_t(T)=1$.
\end{remark}

\begin{proof}
Since $t\ge 2$, then for every vertex $y\ne x$, there exist at least $t-1$
vertices $v_1,...,v_{t-1}\in V(T)-\{x,y\}$ such that
$d(y,x)=d(v_1,x)=...=d(v_{\phi(T)-1},x)$. Moreover, there exists a vertex $y'$
such that there are exactly $t-1$ different vertices satisfying the above
mentioned. Thus, $\{x\}$ is a $t$-metric antiresolving set and, as a
consequence, $\adim_t(T)=1$, since on the other hand, $\adim_t(G)\ge 1$ for any
graph $G$.
\end{proof}

\begin{corollary}
For any tree $T$ such that $\phi(T)\ge 2$, $\adim_{\phi(T)}(T)=1$.
\end{corollary}

According to the results above, it remains to study the $k$-metric dimension of
trees for the case in which every vertex $v$ of $T$ satisfies that $\phi(v)\ne
k$.  To do so, we need to introduce some notations.

We denote by $\Xi_k(T)$, for some $k\in \{2,...,\xi(T)\}$, the set of vertices
$v\in V(T)$ such that $\xi(v)\ge k$. Now, for every $v\in \Xi_k(T)$, let
$$N_{<l_{\xi}}(v)=\{x\in V(T_{v,u})\,:\,u\in N(v)\mbox{ and
}\xi(T_{v,u})<l_{\xi}(v)\}-\{v\},$$
and for the set of vertices $u\in N(v)$ such that $\xi(T_{v,u})=l_{\xi}(v)$,
let $N_{=l_{\xi}}^k(v)$ be the maximum cardinality among all possible sets
obtained as the union of $k$ vertex sets of the branches $T_{v,u}$ where $u\in
N(v)$ minus the vertex $v$ itself. As an example, we consider the tree of
Figure \ref{example}. For $k=2$, there we have that $\Xi_k(T)=\{v_3,v_{15}\}$,
$N_{<l_{\xi}}(v_3)=\{v_{12}\}$, $N_{<l_{\xi}}(v_{15})=\emptyset$,
$N_{=l_{\xi}}^2(v_3)=\{v_1,v_2,v_8,v_9\}$ and
$N_{=l_{\xi}}^2(v_{15})=\{v_{13},v_{14},v_{16},v_{17}\}$ (notice that
$N_{=l_{\xi}}^2(v_{15})$ can be different from this set, but it always has five
vertices).

With this definition we are able to present the following result, where we
analyze only those graphs being $k'$-metric antidimensional for
$k'=\max\{\phi(T),\xi(T)\}$.

\begin{theorem}\label{dimension-trees-bound}
Let $T$ be a $k'$-metric antidimensional of order $n$ with
$k'=\max\{\phi(T),\xi(T)\}$. Then for any $k\le k'$,
$$\adim_k(T)\le
n-\left|\bigcup_{v\in\Xi_i(T)}N_{<l_{\xi}}(v)\right|-\left|\bigcup_{v\in
\Xi_i(T)}N_{=l_{\xi}}^k(v)\right|.$$
\end{theorem}

\begin{proof}
We consider a set $S\subset V(T)$ given by
$$S=V(T)-\bigcup_{v\in \Xi_k(T)}N_{<l_{\xi}}(v)-\bigcup_{v\in
\Xi_k(T)}N_{=l_{\xi}}^k(v).$$
In this sense, for any vertex $x\notin S$, there exists at least $k-1$ vertices
$y_1,y_2,...,y_{k-1}$ not in $S$ such that $d(x,w)=d(y_1,w)=...=d(y_{k-1},w)$
for every $w\in S$. Moreover, if there exists at least one vertex $x'\notin S$
for which there are exactly $k-1$ vertices not in $S$ satisfying the above
mentioned, then $S$ is a $k$-metric antiresolving set and the result follows
since the cardinality of $S$ is given by the formula of the theorem. On the
contrary, if such a vertex does not exist, then $S$ is a $k''$-metric
antiresolving set for $G$ for some $k''\ge k$. Since in this case,
$\adim_{k''}(G)\ge \adim_k(G)$ we obtain the result.
\end{proof}

Consider now the example of Figure \ref{example}. According to the result
above, we have that the set
$S=\{v_3,v_4,v_5,v_6,v_7,v_{10},v_{11},v_{15},v_{18},v_{19}\}$ is a $2$-metric
antiresolving set for such a tree $T$ and $\adim_2(T)\le 10$. Nevertheless,
since $\phi(v_5)=2$, from Remark \ref{remark-phi-equal-1} we have that
$\adim_2(T)=1$. Next we present a family of trees, where the bound of Theorem
\ref{dimension-trees-bound} is achieved.

We consider the family $\mathcal{F}$ of trees $T_r$ satisfying the following
conditions.
\begin{itemize}
\item The center of $T_r$ is formed by two adjacent vertices, say $x,y$.
\item $T_r$ is ``rooted'' in $x,y$.
\item $T_r$ is a complete $r$-ary tree (each vertex of degree greater than one
has $r$ children)
\item Any two leaves being descendants of the same root ($x$ or $y$), have the
same distance to this root.
\end{itemize}
An example of a tree $T_3$ of the family $\mathcal{F}$ is given in Figure
\ref{example4}.

\begin{figure}
\centering
     \includegraphics[width=0.95\textwidth]{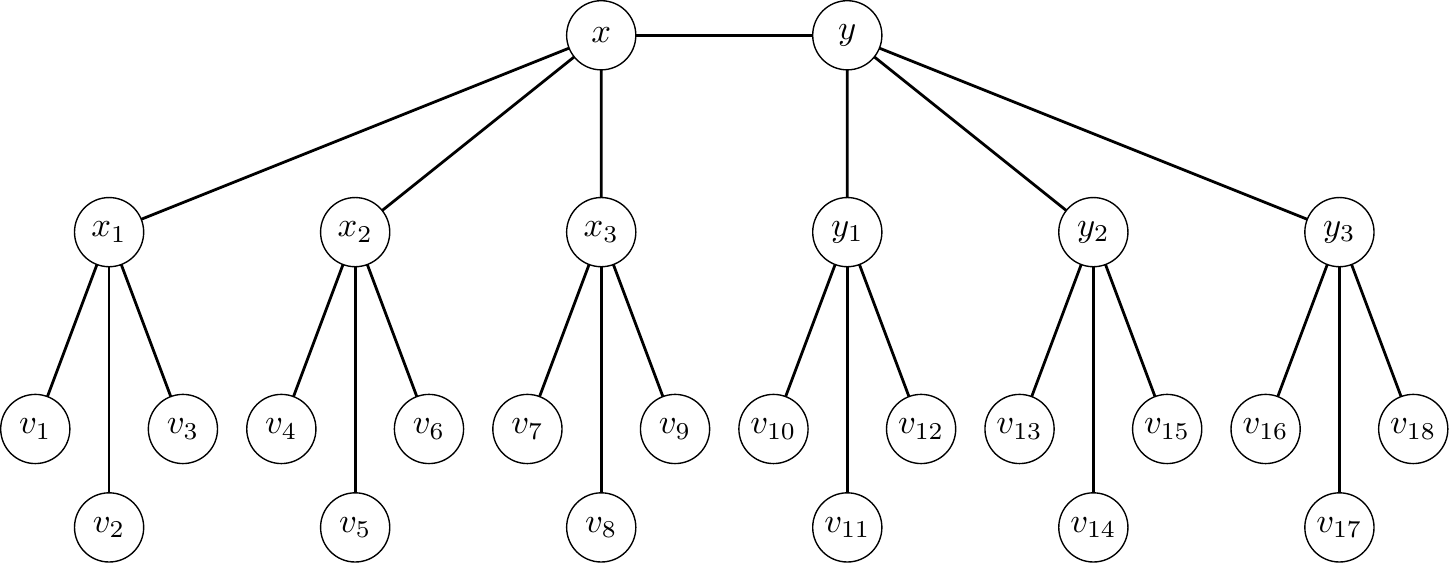}
\caption{A tree $T_3$ of the family $\mathcal{F}$. The set $\{x,y\}$ is a
$3$-metric antiresolving set of minimum cardinality.}
\label{example4}
\end{figure}

It is straightforward to observe that $\xi(T_r)=r$ and $\phi(T_r)=r+1$. Thus,
$T_r$ is $k'$-metric antidimensional for some $k\ge r+1$, and by Remark
\ref{remark-Delta} we have that $k=r+1$. Now on, we compute the $r$-metric
antidimension of $T_r$.
According to the construction of the family $\mathcal{F}$, we see that the root
vertices $x,y$ of a tree $T_r\in \mathcal{F}$ satisfy that $x,y\in \Xi_i(T_r)$.
Also, $N_{<l_{\xi}}(x)=N_{<l_{\xi}}(y)=\emptyset$ and the sets
$N_{=l_{\xi}}^r(x)$, $N_{=l_{\xi}}^r(y)$ are formed by the set of all their
corresponding descendants (this fact makes unnecessary to consider
other vertices of $T_r$). As a consequence of this, by Theorem
\ref{dimension-trees-bound} we have that $\{x,y\}$ is a $r$-metric
antiresolving set and $\adim_r(T_r)\le 2$. Since, for any non-leaf vertex
$u$ of $T_r$ satisfies that $\phi(u)=4$,
we have that any singleton vertex (being not a leaf) is a $(r+1)$-metric
antiresolving set. Thus, $\adim_r(T_r)\ge 2$ and we have that $\adim_r(T_r)=2$,
which makes that the bound of Theorem \ref{dimension-trees-bound} is tight.

\section{Discussion and conclusions}\label{sec_conclusions}

In this article we have introduced a new problem in Graph Theory (the $k$-metric
antidimension problem) that resembles
to the well-known metric dimension problem. The $k$-metric
antidimension is the basis of our novel privacy measure $(k,
\ell)$-anonymity. This measures quantifies the level of privacy offered by an
outsourced social graph against active attacks. Consequently,
privacy-preserving methods
for the publication of social networks ought to consider $(k, \ell)$-anonymity
as
one of their privacy goal.

We have proposed a true-biased algorithm aimed at finding both a
$k$-antiresolving set and a $k$-antiresolving basis in a graph. The algorithm,
although computationally demanding, reached a success rate above $80 \%$ during
the executed experiments when looking for a $k$-antiresolving basis. We expect
future experiments to be conducted over real-life social graphs so that
privacy-preserving methods satisfying $(k, \ell)$-anonymity can be
empirically evaluated in terms of utility and resistance to active attacks.

We have also began the study of mathematical properties of the
$k$-antiresolving sets and the $k$-metric antidimension of graphs. We have
studied some particular graph families like cycles, paths, complete bipartite
graphs and trees. For instance, we have obtained that for any path $P_n$ of odd
order, $\adim_2(P_n)=1$ and for any cycle $C_n$ it follows that
$\adim_2(C_n)=1$ if $n$ is odd, and $\adim_2(C_n)=2$ if $n$ is even. Also, for
every complete bipartite graph $K_{r,t}$, $\adim_k(K_{r,t})=r+t-k$ if $t< k\le
r$, and  $\adim_k(K_{r,t})=r+t-2k$ if $1< k\le t$. For the case of trees we
have presented a tight lower bound for its $k$-metric antidimension in terms of
the order of the tree and the order of some subtrees satisfying some specific
conditions. We have also described an infinite family of $k$-ary trees which
achieve this bound.

Finally, this article opens new and challenging open problems related to the
$k$-metric antidimension of graphs and the privacy concept $(k,
\ell)$-anonymity. For instance, it would be interesting to characterize the
family of graphs such that they are $1$-metric antidimensional, as well as
looking for a close relationship between the $k$-metric antidimension and the
$k$-metric dimension of a graph. In particular, those families of graphs that
resemble to social graphs must be considered. The
computational complexity of computing the
$k$-metric antidimension should also be addressed. In case the
problem is NP-complete, efficient heuristics and privacy-preserving methods
need to be developed so as to compute the $k$-metric antidimension and,
ultimately, transform a social graph into a $(k, \ell)$-anonymous graph for
given values of $k$ and $\ell$.

\section*{Bibliography }
\bibliographystyle{abbrv}

\end{document}